\documentclass{aims}
\usepackage{amsmath}
\usepackage{amssymb}
\usepackage{dutchcal}
\usepackage{paralist}
\usepackage{graphics} 
\usepackage{epsfig} 
\usepackage{graphicx}  \usepackage{epstopdf}
\usepackage[colorlinks=true]{hyperref}
\hypersetup{urlcolor=blue, citecolor=red}
\usepackage{tikz}
\usetikzlibrary{shadows}
\usetikzlibrary{external}
\usepackage{pgfplots}
\usepackage{pgf}
  \usepackage{color}
  \definecolor{purple}{rgb}{0.5,0,1}
  \definecolor{orange}{cmyk}{0,0.7,1,0}
  \definecolor{midgrey}{gray}{0.5}
\def\currentvolume{X}
 \def\currentissue{X}
  \def\currentyear{200X}
   \def\currentmonth{XX}
    \def\ppages{X--XX}
     \def\DOI{10.3934/xx.xx.xx.xx}

\newtheorem{theorem}{Theorem}[section]
\newtheorem{corollary}{Corollary}

\newtheorem{lemma}[theorem]{Lemma}
\newtheorem{proposition}{Proposition}

\theoremstyle{definition}

\newtheorem{remark}{Remark}

\newcommand{\la}{\lambda}

\newcommand{\mc}[1]{\mathcal{#1}}
\newcommand{\comment}[1]{}

\newcommand{\md}{\mathrm{d}}

\newcommand{\sem}[1]{\mbox{$({#1}(t))_{t \geq 0}$}}
\newcommand{\mbb}[1]{\mathbb{#1}}

\newcommand{\e}{\epsilon}

\newcommand {\ti}[1]{\widetilde {#1}}

\newcommand{\p}{\partial}
\newcommand{\cl}[2]{\int\limits_{#1}^{#2}}

\newcommand{\mr}[1]{\mathring{#1}}
\newcommand{\esssup}{\textrm{ess}\!\!\!\!\!\!\sup}

\title[Growth-fragmentation equations ] 
      {Growth-fragmentation equations with McKendrick--von Foerster boundary condition}

\author[J. Banasiak, D.W. Poka and S.K. Shindin]{}

 \email{jacek.banasiak@up.ac.za}
 \email{wetsi@aims.ac.za}
 \email{shindins@ukzn.ac.za}

\thanks{The research has been partially supported by the National
Science Centre of Poland Grant 2017/25/B/ST1/00051 and the National Research Foundation of South Africa Grant 82770}

\begin{document}
\maketitle

\centerline{\scshape Jacek Banasiak}
\medskip
{\footnotesize
 \centerline{Department of Mathematics and Applied Mathematics}
   \centerline{University of Pretoria}
   \centerline{ Pretoria, South Africa}
\centerline{\&}
\centerline{Institute of Mathematics}
\centerline{Technical University of \L\'{o}d\'{z}}
\centerline{\L\'{o}d\'{z}, Poland} 
\medskip
\centerline{\scshape David Wetsi Poka}
\medskip
{\footnotesize
 \centerline{Department of Mathematics and Applied Mathematics}
   \centerline{University of Pretoria}
   \centerline{ Pretoria, South Africa}
}\medskip
\centerline{\scshape Sergey Shindin}
\medskip
{\footnotesize
 \centerline{School of Mathematics, Statistics and Computer Science}
   \centerline{University of KwaZulu-Natal}
   \centerline{Durban, South Africa}
}
\bigskip
\centerline{\small The paper is dedicated to Jerry Goldstein on the occasion of his 80th birthday}
\bigskip
 \centerline{(Communicated by the associate editor name)}

\begin{abstract}
The paper concerns the well-posedness and long-term asymptotics of growth--fragmentation equation with unbounded fragmentation rates and McKendrick--von Foerster boundary conditions. We provide three different methods of proving that there is a strongly continuous semigroup solution to the problem and show that it is a compact perturbation of the corresponding semigroup with a homogeneous boundary condition. This allows for transferring the results on the spectral gap available for the later semigroup to the one considered in the paper. We also provide sufficient and necessary conditions for the irreducibility of the semigroup needed to prove that it has asynchronous exponential growth. We conclude the paper by deriving an explicit solution to a special class of growth--fragmentation problems with McKendrick--von Foerster boundary conditions and by finding its Perron eigenpair that determines its long-term behaviour.
\end{abstract}

\section{Introduction}
We consider the continuous fragmentation equation with growth
\begin{subequations}\label{Tsapele}
\begin{equation}
\partial _t u(x,t) =-\partial _x (r(x)u(x,t)) -a(x)u(x,t) + \int _x^{\infty}a(y)b(x,y)u(y,t)dy,\quad \label{eq1}
\end{equation}
complemented by the initial condition
\begin{equation}
u(x,0)=\mr u(x),\label{eqic}
\end{equation}
and the McKendrick--von Foerster boundary conditions
\begin{equation}\label{Beta}
\lim_{x\to0^+}r(x)u(x,t)=\int _0^{\infty}\beta (y)u(y,t)dy,
\end{equation}
\end{subequations}
for $t,x\in \mbb R_+=(0,\infty)$. Here, $u(x,t)$ is the density of particles of mass/size $x\in (0,\infty)$ at time $t$, $a$ is the fragmentation rate,
 $b$ is the fragmentation kernel that describes mass distribution of $x-$size particles spawned by the fragmentation of a size $y$ particle,
$r$ is the transport coefficient that describes the rate of growth of particles of size $x,$ and $\beta $ represents the rate at which the smallest daughter particles enter the population.  We observe that if $\beta>0$, then the boundary condition depends on the solution on the entire interval $\mathbb{R}_+$, making the problem nonlocal.

The main aim of this paper is to establish the existence of a positive strongly continuous semigroup \sem{S_{K_{\beta,m}}} solving \eqref{Tsapele}  in
\[
X_{m}:=L_1(\mathbb{R_+},(1+x^m)dx),\quad m > 1,
\]
 with the norm $\|\cdot\|_m$, and to show that under natural assumptions it has an the asynchronous exponential growth (AEG) property, that is, that there exist positive constants $\varepsilon$ and $c$ such that
\begin{equation}\label{eqasyexpgr0}
\|e^{-s(K_{\beta,m}) t} S_{K_{\beta,m}}(t)f - \mathcal{P}f \|_m \le c e^{-\varepsilon t} \|f\|_m,\quad t\ge 0,\, f\in X_m,
\end{equation}
where $\mc P$ is the projection on the Perron eigenvector of the generator of  \sem{S_{K_{\beta,m}}}. This property is related to the existence of the so-called spectral gap, that is, the existence of the isolated dominant eigenvalue of \sem{S_{K_{\beta,m}}}.

There is a large body of literature dealing with property \eqref{eqasyexpgr0} for solutions of \eqref{Tsapele} in various settings and using different approaches.
The case when $x$ belongs to a bounded interval has been well-understood since the work of Diekmann, Heijmans and  Thieme, \cite{Diek},   though unbounded rates can be tricky, see \cite{BPR}. A comprehensive theory of the classical McKendrick--von Foerster model of population theory (that is, without the fragmentation operator) can be found in \cite{Webb,ian, Klaus1}. All known work on the full growth--fragmentation problem in unbounded state spaces has been done for \eqref{Tsapele} with homogeneous boundary conditions ($\beta = 0$). In particular, many results have been obtained by the General Relative Entropy method, introduced in \cite{MMP}, which in special cases can also give the exponential rate of decay, \cite{PR}.
 Some results have been established by probabilistic methods, see e.g., \cite{BTK, BertW}. In this work, we focus on operator--theoretic methods the foundation of which can be traced to \cite{Sch, BG1, BG2}. The  latter two papers are based on the study of the Perron eigenvector and eigenvalue done in \cite{DJG}. Quantitative estimates of the spectral gap were obtained by means of Harris operators in \cite{CGY}. Much of the analysis of the above papers was hampered by the lack of satisfactory solvability results for \eqref{Tsapele} with unbounded rates. Such a theory in $X_m$ with $m>1$ was established in \cite{Lamb2020} and it allowed for building a comprehensive spectral gap theory for \eqref{Tsapele} with $\beta =0$ in \cite{MKJB},  paving also the way for extending it to the general case, presented in this paper. It is also worthwhile to observe that a parallel theory in the spirit of both \cite{MKJB} and this paper, but for a discrete model, was obtained in \cite{BJSJEE}.

The paper is organized as follows. In Section \ref{secnt} we introduce basic notation, definitions and assumptions used in the paper. As we intend to give a comprehensive panorama of possible approaches to the growth--fragmentation equation with McKendrick--von Foerster boundary conditions, in Section \ref{secMKJB} we present the main steps for establishing  AEG for \eqref{Tsapele} with $\beta =0$ from \cite{MKJB}. The motivation for this review is that, in principle, having proven the existence of \sem{S_{K_{\beta,m}}}, one could redo all calculations of \textit{op. cit.} for $\beta\neq 0$ to arrive at \eqref{eqasyexpgr0} also in this case. We present, however, a shorter way. In Section \ref{secGT} we give three different theorems on the generation of the semigroup \sem{S_{K_{\beta,m}}}. They extend the results of \cite{BL09}, where we considered \eqref{Tsapele} with $\beta\neq 0$ but only in $X_1$, and of \cite{Lamb2020}, where we discussed the problem in arbitrary $X_m,$ but with $\beta=0$. The third theorem, inspired by \cite[Proposition 4.3, Chapter VI]{Klaus1}, shows that, under suitable assumptions, \sem{S_{K_{\beta,m}}} is a compact perturbation of \sem{S_{K_{0,m}}}, allowing thus the direct application of the results of \cite{MKJB} to establish the existence of a spectral gap for the former from that of the later.  In Section \ref{secsg} we present necessary and sufficient conditions for the irreducibility of \sem{S_{K_{\beta,m}}}. It is a completely new result, significantly extending the ones in \cite{MKJB}
or \cite[Theorem 5.2.21]{JWL}. Finally, in Section \ref{seces} we apply the recent methods of explicitly solving the growth--fragmentation equations, developed in \cite{Banasiak2022}, to construct a solution to a class of models with McKendrick--von Foerster boundary conditions, and to find their Perron eigenpair.

\section{Notation and assumptions}\label{secnt}
We re-write \eqref{Tsapele}  in a compact form as
\begin{equation}\label{Tacp}
\partial _t u(x,t)=-\mc T u(x,t) +\mc Fu(x,t), \qquad (x,t)\in \mathbb{R}^2_+.
\end{equation}
The transport and the fragmentation expressions are defined, respectively,  as
\begin{equation}\label{ops}
\begin{aligned}
\mc Tu(x)&=\p_x\big(r(x)u(x)\big),\\
\mc Fu(x)&=\mc Au(x)+\mc Bu(x)=-a(x)u(x)+\int _x^{\infty}a(y)b(x,y)(y)dy,\\
\end{aligned}
\end{equation}
where the partial derivative is understood in the distributional sense and $u$ is a locally integrable function
for which the integral operator is finite in $\mathbb{R}_+$.  As we mentioned above, the problem shall be analyzed in the space $
X_{m}$. The symbols $X_{m,+}$ and $X^{*}_{m}$ are reserved for
the positive cone (generated by the usual a.e. partial order for measurable functions) and the normed dual of $X_m$,
respectively. If we want to keep the duality pairing between $X_m$ and $X_m^*$ as
\begin{equation}\label{duall}
\langle g,f\rangle = \int _0^{\infty} g(x)f(x)dx,
\end{equation}
then $X_m^*$ can be identified with the space of measurable functions $g$ satisfying
\begin{equation}\label{infy}
\parallel g \parallel^*_{m}=\esssup _{x\in(0,\infty)} \frac{|g(x)|}{1+x^m}< \infty.
\end{equation}

In the presence of the transport term involving partial derivative with respect to the state variable and of
the nonlocal boundary condition, solvability of \eqref{Tsapele} is not a straightforward procedure and there are case when the solving semigroup does not exist, e.g., \cite[p. 9]{MKJB}. For our analysis, we shall adopt assumption on the model coefficients from \cite{JA,JWL} with slight modifications.

The transport coefficient $r$ is assumed to be a continuous function on  $\mbb R_+$, satisfying
\begin{equation}\label{RRe}
0< r(x) \le {r_0}(1+x),\quad x\in \mbb R_+,
\end{equation}
for some ${r_0}>0$. Furthermore, the boundary condition \eqref{Beta} makes sense if and only if
the characteristic curves, associated to the transport part of the model, cross the $t$-axis
transversally. Accordingly, we shall assume that $r(x)$ is bounded and is separated away from
zero as $x\to 0^+$. This implies, in particular, that
\begin{equation}\label{g1}
\int _{0^+}\frac{ds}{r(s)}< \infty,
\end{equation}
where $\int _{0^+} $ denotes the integral in some positive neighborhood of $0$.

The fragmentation rate $a$ is assumed to satisfy
\begin{equation}\label{aa}
0\leq  a \in L_{\infty, loc}([0,\infty)).
\end{equation}
Some stronger results can be obtained if $a$ is polynomially bounded, i.e., if for some $p> 0$,
\begin{equation}\label{aa1}
0\leq a(x)\leq a_0(1+x^p), \quad x \in \mbb R_+.
\end{equation}
The case $p=0$ results in a bounded fragmentation operator and is not considered here.
\begin{remark}
The adopted assumptions imply that $1/r,a/r \in L_{1, loc}([0,\infty)).$ Hence, we can define the following functions
\begin{equation}\label{inte}
R(x):=\int ^x _{0}\frac{ds}{r(s)},\quad Q(x):=\int ^x _{0} \frac{a(s)}{r(s)}ds.
\end{equation}
An immediate consequence of \eqref{inte} is that $R$ is strictly increasing and $Q$ is non-decreasing on
$(0,\infty).$ We have the following limits,
\begin{equation}\label{inteLim}
\begin{aligned}
\lim _{x\to 0^+}R(x)=0,&\quad\lim _{x\to \infty}R(x)=\infty,\\
\lim _{x\to 0^+}Q(x)=0,&\quad\lim _{x\to \infty}Q(x)=M_Q.
\end{aligned}
\end{equation}
Typically, $M_Q=\infty$ provided $a$ is unbounded at $\infty$.
\end{remark}

The fragmentation kernel $b,$ describing the distribution of the sizes $x$ of daughter particles spawned by
fragmentation of a parent particle of size $y$, is assumed to be a non-negative measurable function of two
variables satisfying
\begin{equation}\label{bb}
b(x,y)=0 \;\; \mbox{for }\;x>y.
\end{equation}
\begin{remark}\label{rem10}
If we allow the fragmentation rate $a$ to be zero on some set $I\subset \mbb R_+$, that is, if we admit the situation
that clusters of sizes $y\in I$ do not split, then, to be consistent with physics, we should set $b(x,y) = \delta(y-x)$.
To avoid dealing with measures in the equations, we note that the coefficient $b$ appears only in a product with $a$
and hence $a(y)b(x,y) = 0$ for $y\in I.$ However, the formulae below hold if we treat $b(x,y)dx$ as the Dirac measure.
\end{remark}

For $m\geq 0,$ we define
\begin{equation}\label{nm}
n_m(y)=\int ^y_0x^mb(x,y)dx.
\end{equation}
In particular, $n_0(y)$ is the mean number of daughter particles resulting from splitting of a $y-$aggregate and
$n_1(y)$ is their total mass. Hence, if we consider mass conserving fragmentation, we must assume that for any $y>0$
the identity
\begin{equation}\label{nm1}
n_1(y)=y =\int ^y_0xb(x,y)dx
\end{equation}
is satisfied. It is also physically obvious (and can be proved, see \cite[Section 2.2.3.2]{JWL}) that $n_0(y)> 1$
for $y\notin I$. We further assume that there is $l\geq 0$ and $b_0\in \mathbb{R}_+$, such that
\begin{equation}\label{n0}
n_0(x)\le b_0(1+x^l), \quad x\in \mbb R_+.
\end{equation}
Clearly, $b_0\geq 1$.
Using \eqref{nm1}, we find, for $m\neq 0,1$,
\begin{equation}\label{nm3}
\begin{split}
n_m(y) &= \int ^y_0x^mb(x,y)dx \\
&=\int ^y_0x^{m-1}xb(x,y)dx \left\{\begin{array}{llll}\!\! \leq \!\!\!\!&y^{m-1}\int ^y_0xb(x,y)dx=y^m&\!\!\text{if}\!\!& m>1,\vspace{2mm}\\
\!\! \geq \!\!\!\!&y^{m-1}\int ^y_0xb(x,y)dx=y^m&\!\!\text{if}\!\!& m<1.\end{array}\right.
\end{split}
\end{equation}
We further introduce
\begin{equation}\label{Nm}
N_m(y)=y^m-n_m(y).
\end{equation}
Then, based on \eqref{nm3}, we obtain
\begin{equation}\label{NM}
\begin{aligned}
N_m(y)&\geq 0,\quad m>1,\\
N_1(y)&=0, \quad m=1,\\
N_m(y)&\leq 0,\quad 0\le m <1.
\end{aligned}
\end{equation}
The crucial assumption that allows for the proof of the generation theorem is that there exists $m_0>1$ such that
\begin{equation}\label{lim}
\lim _{y\to \infty}\!\!\!\inf \frac{N_{m_0}(y)}{y^{m_0}}>0.
\end{equation}
It follows,  \cite[Section 5.1.7.2]{JWL} or \cite{BJSJEE},   that if \eqref{lim} holds for some $m_0>1$,
then it holds for all $m>1$. Furthermore, \eqref{lim} implies that for any $m>1$ there is a $c_m<1$
and $y_m>0$ such that
\begin{equation}\label{cm}
n_m(y)\le c_my^m, \quad y\geq y_m.
\end{equation}
Assumption \eqref{lim} is satisfied for most commonly used fragmentation kernels such as the power law, homogeneous or separable kernels. However, if the sizes of the daughter particles are distributed close to that of the fragmenting parent and, correspondingly, close to zero, then such a process may not satisfy \eqref{lim}, see  \cite[Section 5.1.7.2]{JWL}.

Finally, we assume that
\begin{equation}\label{beeta}
0\leq \beta \in X_{m}^{*},
\end{equation}
with $\|\beta\|^*_m = \beta_m$.

It should be noted that for $\beta=0$, we get the homogeneous boundary
condition which was considered in \cite[Section 5.2]{JWL} and \cite{Lamb2020, MKJB}. As the main results on the spectral gap are based on the theory developed in \cite{MKJB}, in the following section we briefly recount its main points.

\section{Spectral gap theory of \cite{MKJB}}\label{secMKJB}
Let $Z_{0,m}$ and $K_{0,m}$ be the operator realizations of, respectively, $-\mc T+\mc A$ and $-\mc T+\mc A+\mc B$, subject to \eqref{Beta} with $\beta =0$, generating strongly continuous semigroups \sem{S_{Z_{0,m}}} and \sem{S_{K_{0,m}}} in $X_m$. The first crucial results of \cite{MKJB} are the resolvent estimates, namely, for sufficiently large $\la$ (see \eqref{res0a}), the resolvent $R(\la, Z_{0,m})$ satisfies for $f\in X_m$
  \begin{equation}\label{pest}
\left|[ R(\lambda,Z_{0,m})f](y)\right | \leq \frac{1}{(1+y^m)r(y)}%
\left\Vert f\right\Vert _{m}, \quad y\in \mbb R_+,
\end{equation}%
and
\begin{equation}\label{alarge}
\int_{0}^{+\infty }\left| [R(\lambda,Z_{0,m})f](y)\right| a(y)(1+y^{m})dy\leq \|f\|_{m }.
\end{equation}
If $y\mapsto  \frac{1}{(1+y^m)r(y)} \in X_m$, then \eqref{pest} would mean that $R(\la, Z_{0,m})$ maps the unit ball in $X_m$ into an order interval and, since $X_m$ is an $L_1$ space, $R(\la, Z_{0,m})$ would be a weakly compact operator. Since, however, it does not hold, we only get that the set
 \begin{equation*}
\left\{ \chi_{\left( 0,\varepsilon ^{-1}\right) }R(\lambda,Z_{0,m})f;\
\left\Vert f\right\Vert_m\leq 1\right\},
\end{equation*}
 where $\chi$ is the characteristic function, is weakly compact for any $\e>0$. If now $a$ was uniformly separated from 0 for large $x$, then \eqref{alarge} would give the smallness of $\chi_{\left(\varepsilon ^{-1},\infty \right) }R(\lambda,Z_{0,m})f$ as $\varepsilon \to 0$, uniformly for  $\left\Vert f\right\Vert _m\leq 1,$ and hence, together with the previous observation, the weak compactness of $R(\la, Z_{0,m})$. Otherwise, we assume that the sublevel sets of $a$ are `thin at infinity' in the sense that for any $c>0$,
 \begin{equation}
\int_{1}^{+\infty }\frac{\chi_{\{\tau;\;a<c\}}(\tau )}{r(\tau )}d\tau <+\infty.  \label{Tss}
\end{equation}%
We observe that if
\begin{equation*} \lim_{y\rightarrow +\infty
}a(y)=+\infty,
\end{equation*}
then \eqref{Tss}
is trivially satisfied. This allows for splitting the estimates of $\chi_{\left(\varepsilon ^{-1},\infty \right) }R(\lambda,Z_{0,m})f$ into the parts where $a$ is, respectively, small and large, and hence yields the required weak compactness of  $R(\la, Z_{0,m})$. Since we are in $L_1$ space, \cite[Lemma 14]{MKM1} gives its compactness.

Next, assumption \eqref{lim} allows for the application of the Miyadera--Desch theorem to show that $K_{0,m} = Z_{0,m}+B_m$ with the domain $D(Z_{0,m})$ (where $B_m =\mc B_{D(Z_{0,m})}$, see \eqref{lin2}) is the generator of the full semigroup \sem{S_{K_{0,m}}} and hence, from \begin{equation*}
R( \lambda, K_{0,m})=R(\lambda, Z_{0,m})\sum_{n=0}^{+\infty }\left[
B_mR(\lambda,Z_{0,m})\right] ^{n},
\end{equation*}
$R( \lambda, K_{0,m})$ is compact.

The last step is to transfer some benefits of the compactness of $R( \lambda, K_{0,m})$ to \sem{S_{K_{0,m}}}. We split
\begin{equation*}
Z_{0,m}+B_m= \left( Z_{0,m}+\widehat{B}_m\right) +\overline{B}_m,
\end{equation*}%
where $0\leq \widehat{B}_m\lneqq {B}_m$ and $\overline{B}_m$ is a bounded integral operator on $X_m$ with compactly supported bounded kernel, hence weakly compact. Thus, if \sem{\widehat{S}_m} is the semigroup generated by $Z_{0,m}+\widehat B_m$, then
\begin{equation*}
S_{K_{0,m}}(t)=\widehat{S}_m(t)+\int_{0}^{t}\widehat{S}_m(t-s)\overline{B}_m S_{K_{0,m}}(s)ds
\end{equation*}%
and $\int_{0}^{t}\widehat{S}_m(t-s)\overline{B}_mS_{K_{0,m}}(s)ds$ is a weakly
compact operator (by \cite{Schl, MKM2}). Thus, by \cite{Ka}, we have the equality of the essential radii,
\begin{equation}
r_{ess}(\widehat{S}_m(t))=r_{ess}(S_{K_{0,m}}(t)).  \label{Stability essential radius}
\end{equation}%
On the other hand,
\begin{equation*}
R(\lambda,Z_{0,m}+\widehat{B}_m)\lneqq R(\lambda,K_{0,m}).
\end{equation*}
Under an assumption ensuring that $R(\lambda,K_{0,m})$ is  irreducible, its compactness gives the strict inequality of spectral radii,
\begin{equation*}
r_{\sigma }\left[ R(\lambda, Z_{0,m}+\widehat{B}_m)\right] <r_{\sigma }\left[
R(\lambda,K_{0,m})\right],
\end{equation*}%
by \cite{IVo}.   This implies inequality of spectral bounds of the generators,
\begin{equation}
s(Z_{0,m}+\widehat{B}_m)<s(K_{0,m}),  \label{strict inequality}
\end{equation}%
and hence, in particular,
\begin{equation*}
s(K_{0,m})>-\infty.
\end{equation*}
 Since in  $L^{1}$-spaces the type of a positive semigroup coincides
with the spectral bound of its generator, see e.g., \cite[Theorem 3.37]{JA}, we get
\begin{equation*}
r_{ess}(S_{K_{0,m}}(t))=r_{ess}(\widehat{S}(t))\leq r_{\sigma }(\widehat{S}(t))=
e^{s(Z_{0,m}+\widehat{B}_m)t}<e^{s(K_{0,m})t}=r_{\sigma }(S_{K_{0,m}}(t)).
\end{equation*}
Using again irreducibility of  \sem{S_{K_{0,m}}},  the spectral bound $s(T+B)$ of its generator is its dominant eigenvalue and a simple pole of the resolvent,  see \cite[Corollary 3.16 of Chapter C-III]{Nag}. Furthermore, by \cite[Proposition  3.4 of Chapter VI]{Klaus1}, $s(K_{0,m})$ is a simple eigenvalue, that is, its eigenspace is one-dimensional, which gives the
asynchronous exponential growth property of the semigroup.
\section{Generation theorems}\label{secGT}
The purpose of this section is twofold. First, we discuss classical wellposedness of \eqref{Tsapele}.
Our analysis is based on the observation (see \cite[Section VI.4]{Klaus1}) that in certain functional settings, the
integro-differential operator, defining \eqref{eq1} and equipped with the nonlocal boundary condition \eqref{Beta}
can be viewed as a multiplicative perturbation of the same operator but coupled with the homogeneous boundary data. Accordingly, we show that the classical wellposedness of \eqref{Tsapele} can be deduced directly from
the known generation results for $\beta=0$.

Second, we characterize a connection between the semigroups associated to the homogeneous ($\beta=0$)
and to the inhomogeneous ($\beta\neq 0$) boundary data. As we shall see later in Section~\ref{secsg},
this characterization, combined with the spectral theory of \cite{MKJB},
yields a simple description of the large time dynamics of \eqref{Tsapele}.

\subsection{The transport semigroup}
Here, we consider
\begin{equation}\label{transG}
\partial _t u = -\mc Tu +\mc Au =-\partial _x(ru) -au=:\mc Zu,
\end{equation}
with the initial and boundary conditions given, respectively, by \eqref{eqic} and \eqref{Beta}.
Let $X_m$, with some $m\ge 1$, be fixed. First, we define the maximal realization of $\mc Z$
in $X_m$ by
\begin{subequations}\label{Zm}
\begin{equation}\label{Zma}
Z_{m} =\mc Z|_{D(Z_m)},
\end{equation}
where
\begin{equation}\label{Zmb}
D(Z_m):=\lbrace u\in X_{m}: \p_x(ru), au \in X_{m}    \rbrace.
\end{equation}
\end{subequations}
Next, we introduce the operator $Z_{\beta,m}$, $0\le \beta \in X^\ast_m$, as the restriction of $Z_m$ to the domain
\begin{equation}\label{bcb}
D(Z_{\beta,m}):=\left\{ u\in D(Z_m): \lim _{x\to 0^+}r(x)u(x)=\int _0 ^{\infty}\beta(y)u(y)dy  \right\},
\end{equation}
defining the relevant boundary condition \eqref{Beta}.
In particular, functions $u\in D(Z_{0,m})$ satisfy the homogeneous boundary condition
\begin{equation}\label{bc0}
\lim _{x\to 0^+}r(x)u(x)=0.
\end{equation}
We note that \eqref{Zmb} implies that $ru$ is absolutely continuous in $\mathbb{R}_+$ and hence
the boundary conditions in \eqref{bcb} and \eqref{bc0} are well-defined.

The resolvent equation, associated to $(Z_{\beta,m}, D(Z_{\beta,m}))$, is given by
\begin{equation}\label{nn1}
\lambda u + \partial _x(ru) + au= f,
\end{equation}
with the boundary condition given in \eqref{bcb}.
The general solution to \eqref{nn1} is
\begin{equation}\label{RT}
u(x)=\frac{e^{-\lambda R(x) -Q(x)}}{r(x)}\int _0^xf(y)e^{\lambda R(y) +Q(y)}dy + C\frac{e^{-\lambda R(x)-Q(x)}}{r(x)},
\end{equation}
where $C$ is an arbitrary constant.
The special case of $\beta = 0$ is settled in \cite{Lamb2020,MKJB}, where it is shown that the resolvent of
$(Z_{0,m}, D(Z_{0,m}))$ is given explicitly by
\begin{subequations}\label{res0prop}
\begin{equation}\label{res0a}
u(x)=[R(\la,Z_{0,m})f](x) = \frac{e^{-\lambda R(x) -Q(x)}}{r(x)}\int _0^xf(y)e^{\lambda R(y) +Q(y)}dy,
\quad \la >\omega_{r,m}:=2mr_0,
\end{equation}
see \eqref{RRe}, and satisfies
\begin{equation}\label{res0}
\parallel R(\lambda,Z_{0,m})f \parallel _{m} \le \frac{1}{\lambda -\omega_{r,m}}\|f\|_m.
\end{equation}
\end{subequations}
The analysis of the general scenario $\beta \neq 0$ relies on two technical results.
\begin{lemma}\cite{Lamb2020}\label{Lema}
Let $m\geq 1$ be fixed. If \eqref{RRe}, \eqref{g1}, \eqref{aa} are satisfied and
$\lambda > \omega_{r,m}:=2mr_0$, then\\
(a) \label{First} for any $0<a<b\le \infty$,
 \begin{equation}\label{Pm1}
P_{m,1}(a,b)=\int _a^{b}\frac{e^{-\lambda R(s)}}{r(s)}(1+s^m)ds
\le \frac{1}{\lambda -\omega_{r,m}}e^{-\lambda R(a)}(1+a^m),
\end{equation}
(b) \label{Second} for any $0<a<b\le \infty$,
\begin{equation}\label{pm2}
P_{m,2}(a,b)=\int _a^{b}\frac{(\lambda +a(s))e^{-\lambda R(s) -Q(s)}}{r(s)}(1+s^m)ds
\le \frac{\lambda}{\lambda -\omega_{r,m}}e^{-\lambda R(a)-Q(a)}(1+a^m).
\end{equation}
\end{lemma}
The second result is an analogue of \cite[Lemma 4.2, p. 218]{Klaus1}, adapted to our settings.
\begin{lemma}\label{lmproj}
Assume \eqref{RRe}, \eqref{g1}, \eqref{aa} and \eqref{beeta} are satisfied. Then the the operator
\begin{subequations}\label{eqproj}
\begin{equation}\label{eqproja}
E_\lambda = e_\lambda\frac{\langle\beta,\cdot\rangle}{1-\langle\beta,e_\lambda\rangle},
\quad e_\la(x) := \frac{d_{\lambda}(x)}{r(x)}:=\frac{e^{-\lambda R(x) -Q(x)}}{r(x)},
\quad \lambda>\omega_{r,m}+\beta_m,
\end{equation}
is well defined and
\begin{equation}\label{eqprojb}
\|E_\lambda f\|_m \le \frac{\beta_m}{\lambda - \omega_{r,m} - \beta_m}\|f\|_m,
\quad f\in X_m, \quad  \lambda>\omega_{r,m}+\beta_m.
\end{equation}
\end{subequations}
Furthermore, for $\lambda>\omega_{r,m}+\beta_m$, the following holds
\begin{subequations}\label{projprop}
\begin{align}
\label{projpropa}
&(I+E_\lambda)^{-1} = I - e_\lambda\langle\beta,\cdot \rangle,\\
\label{projpropb}
&(I+E_\lambda) D(Z_{0,\beta}) = D(Z_{m,\beta}),\\
\label{projpropc}
&\lambda I - Z_{0,m} =  (\lambda I - Z_{\beta,m})(I+E_\lambda).
\end{align}
\end{subequations}
\end{lemma}
\begin{proof}
On the account of \eqref{Pm1}, we have $\|e_\lambda\|_m \le \frac{1}{\lambda - \omega_{r,m}}$. Hence,
for $\lambda>\omega_{r,m}+\beta_m$, the denominator in \eqref{eqproja} is strictly positive and
\eqref{eqprojb} follows. The remainder of the proof is a verbatim repetition of the arguments from
\cite[Lemma VI.4.2]{Klaus1} and is omitted.
\end{proof}

Lemma~\ref{lmproj} indicates that $(Z_{\beta,m}, D(Z_{\beta,m}))$ can be viewed as a multiplicative perturbation of
$(Z_{0,m}, D(Z_{0,m}))$. This fact, combined with the estimates of Lemma~\ref{Lema}, yields a
simple characterization of the resolvent $R(\lambda,Z_{\beta,m})$.
\begin{theorem}\label{resG}
Let the coefficients $r,a$ and $b$ satisfy \eqref{RRe}, \eqref{g1}, \eqref{aa} and \eqref{beeta}.
Then, for any $m\geq 1$ and $\lambda > \omega_{r,m} +\beta_m$, the resolvent of
$(Z_{\beta,m}, D(Z_{\beta,m}))$ is given  explicitly by
\begin{subequations}\label{resolprop}
\begin{equation}\label{RESOL}
\begin{split}
\bigl[R(\lambda,Z_{\beta,m})f\bigr](x)
&= \bigl[(I+E_\lambda)R(\lambda,Z_{0,m})f\bigr](x)\\
&=\bigl[R(\lambda,Z_{0,m})f\bigr](x)+
e_\la(x)\frac{\langle\beta ,R(\lambda,Z_{0,m})f\rangle}{1-\langle\beta ,e_\la\rangle},\qquad f\in X_m.
\end{split}
\end{equation}
Furthermore, the following estimate holds
\begin{equation}\label{RESOLb}
\| R(\lambda ,Z_{\beta,m})f \|_{m}\le \frac{1}{\lambda -\omega_{r,m} - \beta_m}\|f\|_m, \quad
\lambda > \omega_{r,m} +\beta_m.
\end{equation}
\end{subequations}
\end{theorem}
\begin{proof}
Formula \eqref{RESOL} follows directly from \eqref{projpropb} and \eqref{projpropc}.
Further, on the account of \eqref{RESOL} and \eqref{eqprojb}, for $f\in X_m$, we have
\[
\|R(\lambda ,Z_{\beta,m})f\|_m
= \|(I+E_\lambda)R(\lambda ,Z_{0,m})f\|_m
\le \frac{\lambda - \omega_{r,m}}{\lambda - \omega_{r,m}-\beta_m}
\|R(\lambda,Z_{0,m})f\|_m
\]
and \eqref{RESOLb} is the direct consequence of \eqref{res0}.
\end{proof}
Since the operators $R(\lambda, Z_{0,m}),E_\lambda\in \mathcal{L}(X_m)$, $m\ge 1$,
are positive (see formulae \eqref{res0a} and \eqref{eqproja}, respectively), we immediately obtain
\begin{theorem}\label{thtranspgen}
For any given $m\geq 1$, the operator $(Z_{\beta,m},D(Z_{\beta,m}))$ generates a strongly continuous positive semigroup, say $\left( G_{Z_{\beta,m}}(t)\right)_{t\geq 0}$, on $X_{m}$.
\end{theorem}

\subsection{The growth--fragmentation semigroup}\label{transfragG}
Recalling notation \eqref{n0}, we assume
\begin{equation}
\begin{split}
m >1 & \quad \text{if} \quad 0\leq l\leq 1,\\
m\geq l&\quad \text{if}\quad l>1.
\end{split}
\label{mass}
\end{equation}
Next, we consider the full equation \eqref{Tacp}, written in the operator form in $X_m$ as
\begin{equation}\label{lin2}
\p_t u= Z_{\beta,m}u + B_{m}u,  \quad t>0,
\end{equation}
where $B_m$ is the restriction of $\mc B$ to $D(Z_{\beta,m})$, see \cite[Lemma 5.1.4]{JWL}. Under \eqref{mass}, this restriction
is well-defined as $D(Z_{\beta,m})\subset D(A_m):=\{u\in X_m\,:\,au \in X_m\}$ and,  by \eqref{nm}
and \eqref{n0}, for $0\leq u \in D(A_m),$
\begin{equation}\label{Bm}
\begin{split}
\|\mc B u\|_m &= \cl{0}{\infty}\left(\cl{x}{\infty} a(y)b(x,y)u(y)dy\right)(1+x^m)dx
= \cl{0}{y}a(y)u(y)(n_0(y)+n_m(y))dy\\
&\leq 2b_0 \cl{0}{y}a(y)u(y)(1+y^m)dy<\infty.
\end{split}
\end{equation}
We mention that the solvability of \eqref{lin2} with $\beta = 0$ is completely settled in
\cite[Theorem 2.2]{Lamb2020}, where it is shown that $(B_m, D(A_m))$ is a positive Miyadera perturbation of
$(Z_{0,m}, D(Z_{0,m}))$, i.e.,
\begin{equation}\label{demi0}
\|B_m R(\lambda, Z_{0,m})f\|_m \le c_{0,m}\|f\|_m,\quad 0< c_{0,m} <1,\quad \omega_{r,m}<\lambda_0<\lambda,
\end{equation}
for sufficiently large values of $\lambda_0$. Using this fact, together with Lemmas~\ref{Lema}--\ref{lmproj}, we obtain
\begin{theorem}\label{LIN2}
Let \eqref{RRe}, \eqref{g1}, \eqref{aa}, \eqref{beeta} and \eqref{mass} be satisfied.
Then $(B_{m}, D(A_{m}))$ is a Miyadera perturbation of $(Z_{\beta,m}, D(Z_{\beta, m}))$,
and hence $(K_{\beta,m},D(Z_{\beta,m})):=(Z_{\beta,m} + B_{m},D(Z_{\beta,m}))$
generates a positive $C_0$-semigroup, say $(S_{K_{\beta,m}}(t))_{t\geq 0}$, in $X_{m}$.
Furthermore,
\begin{equation}\label{eqsimb}
S_{K_{\beta,m}}(t) \ge S_{K_{0,m}}(t), \quad t\ge 0.
\end{equation}
\end{theorem}
\begin{proof}
As in \cite{Lamb2020}, the proof is based on the Desch's version of the Miyadera perturbation theorem,
\cite[Lemma 5.12]{JA}.

Since both, $B_m$ and $R(\lambda, Z_{\beta,m})$ are positive and the norm $\|\cdot\|_m$ is additive
in $X_{m,+}$, we only need to show
\begin{equation}
\|B_{m}R(\lambda, Z_{\beta,m}) f\|_{m}\le  c\|f\|_m,\quad 0<c<1,
\end{equation}
for sufficiently large values of $\lambda>\omega_{r,m}+\beta_m$ and $f\geq 0$. By \eqref{RESOL} and \eqref{demi0},
we have
\begin{align*}
\|B_mR(\lambda, Z_{\beta,m}) f\|_m
&\le \| B_m R(\lambda, Z_{0,m}) f \|_m
+ \|B_m E_\lambda R(\lambda, Z_{0,m}) f\|_m\\
&\le c_{0,m}\|f\|_m + \frac{\bigr|\langle\beta, R(\lambda, Z_{0,m}) f \rangle\bigl|}{1-\langle\beta, e_\lambda\rangle}
\|B_m e_\lambda\|_m\\
&\le \left(c_{0,m}  + \frac{\beta_m \|B_m e_\lambda\|_m}{\lambda-\omega_{m,r} - \beta_m} \right) \|f\|_m
=: (c_{0,m} + c_{\beta,m}) \|f\|_m.
\end{align*}
The inclusion $e_\lambda \in D(A_m)$ and bounds \eqref{pm2} and \eqref{Bm} imply that
\[
0< c_{\beta,m} \le \frac{2b_0\beta_m\lambda}{(\lambda-\omega_r)(\lambda-\omega_r-\beta_m)}\to 0,
\]
as $\lambda\to\infty$. Hence, for $\lambda$ sufficiently large, $0<c_{0,m}+c_{\beta,m}<1$ and
$(B_m, D(Z_{\beta,m}))$ is the Miyadera perturbation of $(Z_{\beta,m}, D(A_{m}))$.
By Desch's theorem, $(Z_{\beta,m}+B_{m},D(Z_{\beta,m}))=:(K_{\beta,m},D(Z_{\beta,m}))$
generates a positive semigroup in $X_{m}$, say $(S_{K_{\beta,m}}(t))_{t\geq 0}$.

To prove \eqref{eqsimb}, on the account of Hille's identity (see e.g., \cite[Corollary 5.5, p.~223]{Klaus}),
it is sufficient to verify that $R(\lambda, K_{\beta,m})\ge R(\lambda, K_{0,m})$ for sufficiently large $\lambda$.
Since $(B_m, D(A_m))$ is the Miyadera perturbation of both $(Z_{0,m}, D(Z_{0,m}))$
and $(Z_{\beta,m}, D(Z_{\beta,m}))$, we have,  by \cite[Theorem 5.10]{JA},
\begin{equation}
R(\lambda, K_{\alpha,m}) = R(\lambda,Z_{\alpha,m})\sum\limits_{n=0}^{\infty} (B_mR(\lambda,Z_{\alpha,m}))^n
\label{reseq}
\end{equation}
for $\alpha =0, \beta.$
Since, by \eqref{RESOL}, $R(\lambda,Z_{\beta ,m})\ge R(\lambda,Z_{0,m})\geq 0$ for
$\lambda>\omega_{r,m}+\beta_m$   and $B_m\geq 0$, we immediately get   $R(\lambda, K_{\beta,m})\geq R(\lambda, K_{0,m})$ for large $\lambda$.
\end{proof}

As mentioned in Section~\ref{secnt}, under the additional assumption \eqref{aa1},
\[
0\leq a(x)\leq a_0(1+x^p), \quad x \in \mbb R_+,
\]
we can prove that $(S_{K_{\beta,m}}(t))_{t\geq 0}$ is quasi-contractive in $X_m$, provided
\begin{equation}
\begin{split}
m >1 & \quad \text{if} \quad 0\leq l+p\leq 1,\\
m\geq l+p&\quad \text{if}\quad l+p>1.
\end{split}
\label{massum}
\end{equation}
with $p$ and $l$ defined in (\ref{aa1}) and \eqref{n0}, respectively.
Let $w_m(x):=1+x^m$.

\begin{lemma}\label{est10}
Let \eqref{RRe}, \eqref{aa}, \eqref{nm1} and \eqref{beeta} be satisfied. Then,  for $u \in D(Z_{\beta,m})$, we have
\begin{equation}\begin{split}
&\cl{0}{\infty}[(Z_{\beta,m}+B_{m})u](x)w_m(x){d}x
= -c_{m}(u) \\
&:=\cl{0}{\infty}\beta(x)u(x)dx +
 m\cl{0}{\infty}r(x)u(x)x^{m-1} {d}x
\!-\cl{0}{\infty}(N_0(x)+N_m(x))a(x)u(x) {d}x. \label{subfuncta'}\end{split}
\end{equation}
\end{lemma}
\begin{proof} We have
\begin{align*}
&\cl{0}{\infty}\!\![(Z_{\beta,m}+B_{m})u](x)w_m(x)d x\\
&\phantom{xxxxxxx}=\!-\cl{0}{\infty}\!\!\p_x(r(x)u(x))w_m(x) dx -\cl{0}{\infty}\!\!a(x)(x)w_m(x)dx
+
\cl{0}{\infty}[B_mu](x)w_m(x)d x\\
&\phantom{xxxxxxx}= -\cl{0}{\infty}\!\!\p_x(r(x)u(x))w_m(x) dx
-\cl{0}{\infty}(N_0(x)+N_m(x))a(x)u(x)dx,
\end{align*}
where  we used (\ref{Bm}) and \eqref{Nm} to  obtain  the last line.
Also, for $0<x_0<x_1<\infty$
\begin{equation}
\cl{x_0}{x_1}\p_x(r(x)u(x))w_m(x) d x=
r(x_1)u(x_1)w_m(x_1)-r(x_0)u(x_0)w_m(x_0) - m\cl{x_0}{x_1}r(x)u(x)x^{m-1}d x.
\label{intx}
\end{equation}
The integral on the LHS converges by the definition of $D(Z_{\beta,m})$ and on the RHS converges by \eqref{RRe}.
By \eqref{bcb}, we have
\[
\lim\limits_{x_0\to 0^+} r(x_0)u(x_0)=\cl{0}{\infty}\beta(x)u(x)dx.
\]
and, since the limit is finite,
\[
\lim\limits_{x_0\to 0^+} r(x_0)u(x_0)(1+x_0^m) = \lim\limits_{x_0\to 0^+} r(x_0)u(x_0)
+ \lim\limits_{x_0\to 0^+} r(x_0)u(x_0)x_0^m = \cl{0}{\infty}\beta(x)u(x)dx.
\]
Thus, there exists
\[
\lim\limits_{x_1\to \infty}r(x_1)u(x_1)(1+x_1^m)= L\geq 0.
\]
If $L>0$, then there is $0<L'<L$ such that $r(x)u(x)w_m(x)\geq L'$ for large $x$, but then
 $u\notin X_m$. Thus
\begin{equation}
\lim\limits_{x_1\to \infty}r(x_1)u(x_1)w_m(x_1)= 0,
\label{rinf}
\end{equation}
and \eqref{subfuncta'} is proved.
\end{proof}

\begin{proposition}\label{prop10}
Assume \eqref{RRe}, \eqref{aa},  \eqref{aa1}, \eqref{nm1} and \eqref{beeta} are satisfied.
Then $(S_{K_{\beta,m}}(t))_{t\geq 0}$ is a positive quasi-contractive $C_0$-semigroup with type not exceeding
\begin{equation}
\omega_{\beta,m} :=  \beta_m+\omega_{r,m} +4a_0b_0.
 \label{omegamin}
 \end{equation}
 \end{proposition}
\begin{proof}
For $u \in D(Z_{\beta,m})_+$, by \eqref{subfuncta'}, (\ref{Nm}), \eqref{RRe} and \eqref{aa1},
\begin{align*}
-c_m(u)&\leq  \cl{0}{\infty}\beta(x)u(x)dx +
 m\cl{0}{\infty}r(x)u(x)x^{m-1} {d}x
\!-\cl{0}{\infty}N_0(x)a(x)u(x) {d}x -\cl{0}{\infty}N_m(x)a(x)u(x) {d}x \\
&\leq \beta_m \|u\|_{m} +r_0 m \cl{0}{\infty}u(x)(1+x)x^{m-1} \md x
+ 4a_0b_0 \| u\|_m-\cl{0}{\infty}N_m(x)a(x)u(x) {d}x\\
&\leq (\beta_{\infty} + 2 m r_0 + 4a_0b_0)\|u\|_{m}-\cl{0}{\infty}N_m(x)a(x)u(x) {d}x,
\end{align*}
where we used
\[
0\leq -\frac{N_0(x)a(x)}{1+x^m} \leq \frac{n_0(x)a(x)}{1+x^m}\leq a_0b_0\frac{(1+x^p)(1+x^l)}{1+x^m}
\leq 4a_0b_0,
\]
for $m\geq p+l$, see \cite[Lemma 5.1.46]{JWL}.
Then, by e.g., \cite[Proposition 9.29]{JA}, there is a minimal extension $\ti K_{\beta,m}$ of $Z_{\beta,m}+B_m$
generating a positive quasi-contractive semigroup, say \sem{S_{\ti K_{\beta,m}}}, with the growth rate not exceeding  $\omega_{\beta,m}$.
Since, by \cite[Theorem 5.2]{JA},
\[
R(\la, \ti K_{\beta,m}) = \sum\limits_{n=0}^{\infty} R(\la,Z_{\beta,m})(B_mR(\la,Z_{\beta,m}))^n,
\quad \la > \omega_{\beta,m},
\]
and, using \eqref{reseq}, we see  that $K_{\beta,m} = \ti K_{\beta,m}$, \sem{S_{\ti K_{\beta,m}}}=\sem{S_{K_{\beta,m}}} and hence the latter is quasi-contractive.
\end{proof}

\subsection{A characterization of $(S_{K_{\beta,m}}(t))_{t\ge 0}$}\label{subsecCM}

By virtue of Lemma~\ref{lmproj} and subsequent Theorems~\ref{thtranspgen} and
\ref{LIN2}, $((K_{\beta,m}, D(Z_{\beta,m}))$ is a multiplicative
perturbation of $(K_{0,m}, D(Z_{0,m}))$, where the multipliers
$(I+E_\lambda)$ and $(I+E_\lambda)^{-1}$ are compact perturbations of the identity.
This fact, combined with the approach from \cite[Section 6.4]{Klaus1}, yields the following
characterization of $(S_{K_{\beta,m}})_{t\ge0}$, with $\beta\neq 0$.

\begin{theorem}\label{thpert}
Assume that \eqref{RRe}, \eqref{aa1} and \eqref{n0}, with \eqref{massum}, are satisfied. Additionally, let
\begin{subequations}\label{betaas}
\begin{equation}
\beta \in X^*_k, \quad k+p\leq m,\label{betaas1}
\end{equation}
\begin{equation}
\p_x \beta \in X^*_{m-1}.\label{betaas2}
\end{equation}
\end{subequations}
Then $(S_{K_{\beta,m}}(t))_{t\ge0}$ is similar to a compact perturbation of $(S_{K_{0,m}}(t))_{t\ge0}$, i.e.,
\begin{equation}\label{eqsim}
(I+E_\lambda)^{-1}S_{K_{\beta,m}}(t)(I+E_\lambda) - S_{K_{0,m}}(t) \in \mathcal{K}(X_m),\quad t\ge 0.
\end{equation}
\end{theorem}
\begin{proof}
(a) On the account of \eqref{projprop}, the elementary identity
$(\lambda I-Z_m)e_\lambda = 0$ and the inclusions $D(Z_{0,m}), D(Z_{\beta,m}) \subset D(A_m)$, for
$\lambda>\omega_{r,m}+\beta_m$, we have
\begin{equation}\label{sim_1}
\begin{split}
&(I+E_\lambda)^{-1}(\lambda I-K_{\beta,m})(I+E_\lambda)f\\
&\qquad\qquad
= (\lambda I-K_{0,m})f  + e_\lambda\langle \beta, K_{0,m}f\rangle \\
&\qquad\qquad
- \bigl[\lambda e_\lambda\langle\beta,f\rangle + (I+E_\lambda)^{-1}B_m E_\lambda f\bigr],
\quad f \in D(Z_{0,m}).
\end{split}
\end{equation}
Note that $e_\lambda\in D(A_m)$, therefore, by virtue of \eqref{eqproja} and \eqref{pm2}, the rank-one operator
$\bigl[\lambda e_\lambda\langle\beta,f\rangle + (I+E_\lambda)^{-1}B_m E_\lambda f\bigr]$
is bounded in $X_m$ and hence, is compact.

(b) We let $D_{0, m} := e_\lambda\langle \beta, K_{0,m} \cdot\rangle$.
In view of \eqref{betaas}, for $f\in D(Z_{0,m})$, we have
\begin{align*}
\langle \beta,B_mu\rangle &= \cl{0}{\infty}a(y)u(y) \left(\cl{0}{y}b(x,y)\beta(x) dx\right)dy
\leq \beta_k \cl{0}{\infty}a(y)u(y) \left(\cl{0}{y}b(x,y)(1+x^k)dx\right)dy \\
&\leq 2\beta_kb_0a_0 \cl{0}{\infty}u(y)(1+y^m)dy =  2\beta_kb_0a_0\|u\|_m
\end{align*}
and
\begin{align*}
\langle\beta,Z_{0,m}u\rangle &= \cl{0}{\infty} \beta(x) \p_x(r(x)u(x) dx - \cl{0}{\infty} \beta(x)a(x)u(x)dx\\
&= \lim\limits_{x_1\to \infty} \beta(x_1)r(x_1)u(x_1)- \lim\limits_{x_0\to 0^+} \beta(x_0)r(x_0)u(x_0) \\
&- \cl{0}{\infty} \p_x\beta(x) r(x)u(x)dx - \cl{0}{\infty} \beta(x)a(x)u(x)dx\\
&= - \cl{0}{\infty} \p_x\beta(x) r(x)u(x)dx - \cl{0}{\infty} \beta(x)a(x)u(x)dx,
\end{align*}
where \eqref{bc0} and \eqref{rinf} (as $k<m$) are used to show that both limits vanish.
The estimates show that $D_m$ extends to a bounded rank-one (and hence compact)
operator $\bar{D}_m$ in $X_m$.
Thus, from \eqref{sim_1} it follows that $(K_{\beta,m}, D(Z_{\beta,m}))$ is similar to a compact perturbation of
$(K_{0,m}, D(Z_{0,m}))$ and hence, by \cite[Chapter II, Section 2.1]{Klaus1}, generates a semigroup
$(S_1(t))_{t\ge 0}$ satisfying
\[
(I+E_\lambda)^{-1}S_1(t)(I+E_\lambda) - S_{K_{0,m}}(t) \in \mathcal{K}(X_m),\quad t\ge 0.
\]
By the uniqueness of the generator, $(S_1(t))_{t\ge 0}$ must coincide with
$(S_{K_{\beta,m}}(t))_{t\ge 0}$ of Theorem~\ref{LIN2} and \eqref{eqsim} follows.
\end{proof}


\section{Irreducibility of the semigroup}\label{secsg}

An important part in the existence of the spectral gap of a semigroup is played by its irreducibility.
In contrast to the pure fragmentation semigroup, which is not irreducible, here we have an interplay
of the growth mechanism and fragmentation, which reduces the size of particles, and these antagonistic processes, under
natural assumptions on the fragmentation rate $a$ and the daughter particles distribution function $b$,
yield irreducibility of the solution semigroup.

The irreducibility can be ensured by two mechanisms. In general, if we allow $a$ to be zero over some set $I$
(see Remark \ref{rem10}) then, to make sure that there is sufficient flow "downward", we must assume
that the particles that can split, produce daughter particles of sufficiently small sizes so that the process will not be
confined to sizes bigger than that in $I$. This can follow if the gain of particles of smaller sizes, entering the ensemble
after splitting of a $y$-size particle due to the fact that $b(x,y)>0$ for some $x<y$, is sufficiently uniform
to eventually fill up $\mbb R_+$. Otherwise, this can be also ensured by a nonzero $\beta,$ which generates
particles of "zero" size entering the system out from particles of sizes $y \in {\rm supp}\,\beta$.   This corresponds
to classical results from McKendrick--von Foerster equation in population theory, where the model fails to be irreducible
if the old (in this case large) individuals cannot reproduce, see e.g., \cite[Theorem 5.2]{ian}.

To make these physical intuitions precise, we have to introduce some notation. In this case, as a consequence
of \eqref{nm1}, we see that for any $y>0$
\begin{equation}
\emptyset\neq \mathrm{supp}\,b(\cdot,y)\subset [0,y].
\label{suppb}
\end{equation}
Here and elsewhere in the paper, for a measurable function $f$, by $\mathrm{supp}\,f$ we understand the
essential support of $f$, that is, the complement of the set on which $f$ is almost everywhere equal to 0 and by
$\sup$ and $\inf$ we understand the essential supremum and infimum.  Then, for any $y\in\mbb R_+$, we define
$$
\mc b(y) = \inf \mathrm{supp}\,b(\cdot,y).
$$
Note that, in agreement with Remark \ref{rem10},  we decided to use $b$ rather than $ab$ here, as then $\mc b(y) = y$
for $y \notin \mc A:=\mathrm{supp}\, a$ is well-defined, whereas supp\,$ab$ is empty for such $y$s. The function $\mc b$ satisfies
$$
0\leq \mc b(y)<y,
$$
for any $y\in \mc A$.  Next, for any $z>0,$ we define
$$
\mc c(z) = \inf\limits_{y\geq z} \mc b(y).
$$
Then also
$$
0\leq \mc c(z)\leq z.
$$
For any $z_0>0,$ the sequence $(z(n,z_0))_{n\geq 1} = (\mc c^{(n)}(z_0))_{n\geq 1}$, where
$ \mc c^{(n)}(y_0) := \underbrace{\mc c(\mc c(...\mc c(z_0)))}_{n\,\rm times}$ is nonincreasing and bounded
from below and thus has a limit, say $\mc c_\infty(z_0)$. Let
$$
\bar{\mc c} :=\sup_{z_0\in \mbb R_+}\{\mc c_\infty(z_0)\}.
$$
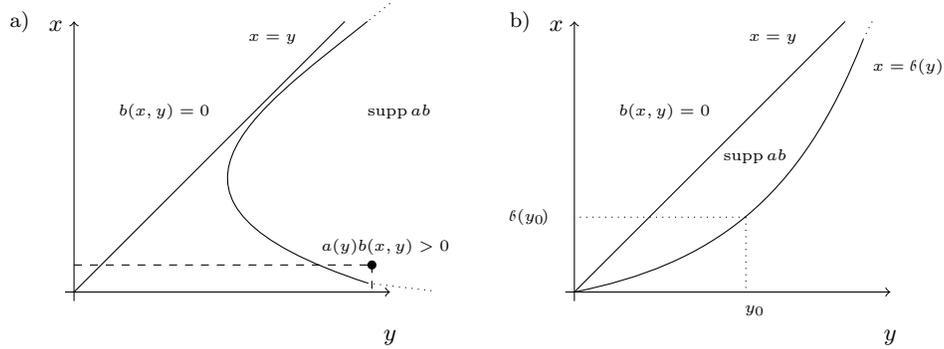
\begin{figure}
\begin{center}
\begin{tikzpicture}[scale=1.2]
\draw (-2.6,3) node {a)};
\draw[thin,->] (-2.1,0)--(1.5,0);
\draw[thin,->] (-2,-0.1)--(-2,3);
\draw[thin,-] (-2,0)--(1,3);
\draw (1.5,-0.5) node {\small{$y$}};
\draw (-2.2,2.95)  node {\small{$x$}};
\draw (-1,2) node {\tiny{$b(x,y)=0$}};
\draw (0.2,2.8) node {\tiny{$x=y$}};
\draw [thin,dotted](1.5,3.2)--(1.25,3);
\draw (1.25,3)..controls (0,2) and (-1.5,1)..(1.25,0.1);
\draw [thin,dotted](1.25,0.1)--(2,0.01);
\draw [thin,dashed] (-2,0.3)--(1.3,0.3);
\draw [thin,dashed] (1.3,0.3)--(1.3,0);
\fill (1.3,0.3) circle [radius=0.05];
\draw (1.45,0.5) node {\tiny{$a(y)b(x,y)>0$}};
\draw (1.6,2) node {\tiny{${\rm supp}\,ab$}};
\end{tikzpicture}\qquad
\begin{tikzpicture}[scale=1.2]
\draw (-2.6,3) node {b)};
\draw[thin,->] (-2.1,0)--(1.5,0);
\draw[thin,->] (-2,-0.1)--(-2,3);
\draw[thin,-] (-2,0)--(1,3);
\draw (-2,0)..controls (-0.5,0.3) and (0.5,1)..(1.2,2.8);
\draw [dotted] (1.2,2.8)--(1.3,3);
\draw (0,1.5) node {\tiny{${\rm supp}\,ab$}};
\draw [thin,dotted] (-0.1,0.83)--(-0.1,0);
\draw [thin,dotted] (-0.1,0.83)--(-2,0.83);
\draw (-2.5,0.83) node {\tiny{$\mc b(y_0\!)$}};
\draw (0,-0.2) node {\tiny{$y_0$}};
\draw (1.5,-0.5) node {\small{$y$}};
\draw (-2.2,2.95) node {\small{$x$}};
\draw (-1,2) node {\tiny{$b(x,y)=0$}};
\draw (1.7,2.5) node {\tiny{$x=\mc b(y)$}};
\draw (0.2,2.8) node {\tiny{$x=y$}};
\end{tikzpicture}
\caption{\footnotesize{Illustration of two cases of supp$\,ab$. On the left, supp$\,ab$ extends to infinity
for any $x>0$, as in Remark \ref{exMKJB}.(a). On the right,  supp$\,ab$ allows to reach $x=0$ by iterations.}}\label{fig1a}
\end{center}
\end{figure}
\begin{lemma}\label{lem2}
 \mbox{}\\
1. If $\mc c(\bar{\mc c})\geq \bar{\mc c}$, then for any $z\geq \bar{\mc c}$ we have $\mc c_\infty (z) = \bar{\mc c}$.\\
2. If $\mc c(\bar{\mc c})< \bar{\mc c}$, then for any $z > \bar{\mc c}$, $\mc c(z)>\bar{\mc c}$  and
$\mc c_\infty (z) = \bar{\mc c}$.\\
Hence,
$$
\bar{\mc c} = \max_{z\in \mbb R_+}\{\mc c_\infty(z)\}
$$
and, in particular, $\bar{\mc c}$ is isolated in $\{\mc c_\infty(z)\}_{z\in \mbb R_+}$
and it can be only approached by sequences $(\mc c^{(n)}(z))_{n\geq 1}$, $z\geq \bar{\mc c},$ from above.
\end{lemma}
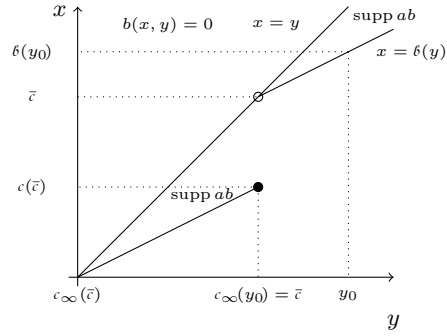
\begin{figure}
\begin{tikzpicture}[scale=1.2]
\draw[thin,->] (-2.1,0)--(1.5,0);
\draw[thin,->] (-2,-0.1)--(-2,3);
\draw[thin,-] (-2,0)--(1,3);
\draw[scale=1, domain=-2:0, smooth, variable=\x, black] plot ({\x}, {0.5*(\x+2))});
\draw[scale=1, domain=0:1.5, smooth, variable=\x, black] plot ({\x}, {0.5*(\x+2) +1)});
\draw (1.5,-0.5) node {\small{$y$}};
\draw (-2.2,2.95) node {\small{$x$}};
\draw (-1,2.8) node {\tiny{$b(x,y)=0$}};
\draw (0.2,2.8) node {\tiny{$x=y$}};
\draw (1.7,2.5) node {\tiny{$x=\mc b(y)$}};
\draw (0,1) circle  [radius=0.05];
\fill (0,1) circle  [radius=0.05];
\draw (0,2) circle  [radius=0.05];
\draw [dotted] (-2,2)--(0,2);
\draw [thin,dotted] (1,2.5)--(1,0);
\draw [thin,dotted] (1,2.5)--(-2,2.5);
\draw (-2.5,2.5) node {\tiny{$\mc b(y_0\!)$}};
\draw (1,-0.2) node {\tiny{$y_0$}};
\draw [thin,dotted] (0,1)--(0,0);
\draw (0,-0.2) node {\tiny{$\mc c_\infty\!(y_0\!)=\bar{\mc c}$}};
\draw [thin,dotted] (0,1)--(-2,1);
\draw (-2.5,2) node {\tiny{$\bar{\mc c}$}};
\draw (-2.5,1) node {\tiny{$\mc c(\bar{\mc c})$}};
\draw (-2,-0.2) node {\tiny{$\mc c_\infty(\bar{\mc c})$}};
\draw (-0.62,0.91) node {\tiny{${\rm supp}\,ab$}};
\draw (1.4,2.9) node {\tiny{${\rm supp}\,ab$}};
\end{tikzpicture}
\caption{The case of Lemma \ref{lem2}.(2). Here $\mc c_\infty(\bar {\mc c}) <\bar{\mc c}$ and  $\mc c_\infty(z) =\bar{\mc c}$ for any $z>\bar{\mc c}$. }\label{fig1c}
\end{figure}
\begin{proof}
In case 1., first we observe that the assumption, and $\mc b(z)\leq z,$ implies $\mc c(\bar{\mc c}) = \bar{\mc c}$.
Further, properties of infimum imply $\mc c(z_1)\leq \mc c(z_2)$, provided $z_1\leq z_2$. From this, it follows that
for any $z\geq \bar{\mc c}$, we have
$$
\mc c^{(n)}(z)\geq \mc c^{(n)}(\bar{\mc c}) = \bar{\mc c},
$$
and hence, $\mc c_\infty(z) \geq \bar{\mc c}$. From the definition of $\bar{\mc c}$, $\mc c_\infty(z) = \bar{\mc c}$.

In case 2., assume that for some $z>\bar{\mc c}$, we have $\mc c(z) \leq \bar{\mc c}$. Then for all $y\leq z$
it holds $\mc c(y) \leq \bar{\mc c}$ and, iterating and using the assumption,
$\mc c^{(n)} (z) <\bar{\mc c}$ for all $n$.  By monotonicity,  $\mc c_\infty(z) <\bar{\mc c}$. Using this observation,
we see that it is impossible to have $c_\infty(x) = \bar{\mc c}$ for any $x>z$. Indeed,  again using the
monotonicity, we would have $\mc c^{(n)}(x) \in [\bar{\mc c},z]$ for sufficiently large $n$. But then, from the first part,
$\mc c_\infty(x) < \bar{\mc c}$, contradicting the hypothesis. Summarizing, by the definition of $\bar{\mc c}$, the only
possible situation in this case is $\mc c(\bar{\mc c})< \bar{\mc c}$ and $\mc c(z)> \bar{\mc c}$ for $z>\bar{\mc c}$,
see Fig. \ref{fig1c}. If we take any such $z$, then $\bar c< \mc c(z)\leq z$, and we see that $(\mc c^{(n)}(z))_{n\geq 1}$
is confined to $[\bar{\mc c}, z]$ and thus converges to $\bar{\mc c}$, by its definition.


This shows that $\bar{\mc c}$ cannot be approached from below by elements of the set
$\{\mc c_\infty(z)\}_{z\in \mbb R_+}$ and since naturally it cannot be approached from above by such elements,
it must be an isolated point and hence the supremum is attained. Since the sequences $(\mc c^{(n)}(z))_{n\geq 1}$
are nonincreasing, $\bar{\mc c}$ can be approximated by such sequences with $z\geq \bar{\mc c}$ only from above.
\end{proof}

\begin{theorem}
Semigroup \sem{S_{K_{\beta,m}}} is irreducible if and only if
\begin{equation}
\sup \mathrm{supp}\, \beta =\infty\label{betas}
\end{equation}
or
\begin{equation}
\sup \mathrm{supp}\, \beta> \bar{\mc c}
\label{ass2a}
\end{equation}
or
\begin{equation}
\bar{\mc c} =0.
\label{ass2}
\end{equation}\label{thirred}
\end{theorem}
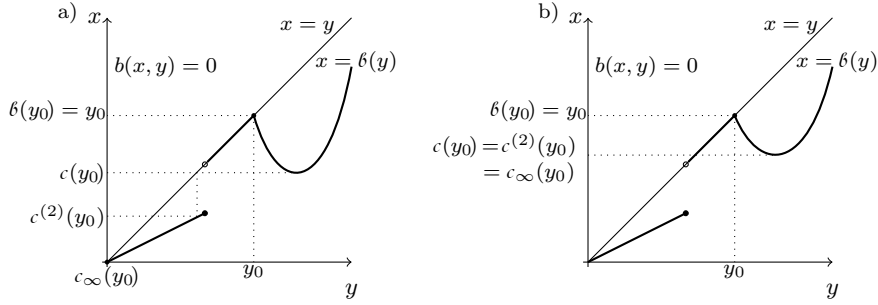
\begin{figure}
\begin{center}
\begin{tikzpicture}[scale=0.65]
\draw (-2.8,5.1) node {a)};
\draw[thin,->] (-2.1,0)--(3,0);
\draw[thin,->] (-2,-0.1)--(-2,5);
\draw[thin,-] (-2,0)--(3,5);
\draw[-,thick] (0.05,2.05)--(1,3);
\draw[scale=1, domain=-2:0, smooth, variable=\x, black,thick] plot ({\x}, {0.5*(\x+2))});
\draw  [thick] (1,3)..controls (1.5,1.3) and (2.5,1.3)..(3,4);
\draw (3,-0.6) node {\small{$y$}};
\draw (-2.25,5) node {\small{$x$}};
\draw (-0.8,4) node {\footnotesize{$b(x,y)=0$}};
\draw (2.1,4.8) node {\footnotesize{$x=y$}};
\draw (3.1,4.1) node {\footnotesize{$x=\mc b(y)$}};
\draw (0,1) circle  [radius=0.05];
\fill (0,1) circle  [radius=0.05];
\draw (0,2) circle  [radius=0.05];
\fill (1,3) circle  [radius=0.05];
\fill (-2,0) circle  [radius=0.05];
\draw [thin,dotted] (1,3)--(1,0);
\draw [thin,dotted] (-2,1.83)--(1.9,1.83);
\draw [thin,dotted] (1,3)--(-2,3);
\draw (-3,3.1) node {\footnotesize{$\mc b(y_0\!)=y_0$}};
\draw (-2.5,1.8) node {\footnotesize{$\mc c(y_0\!)$}};
\draw (1,-0.2) node {\footnotesize{$y_0$}};
\draw [thin,dotted] (-0.16,0.94)--(-2,0.94);
\draw [thin,dotted] (-0.16,1.84)--(-0.16,0.94);
\draw (-2.8,0.94) node {\footnotesize{$\mc c^{(2)}(y_0)$}};
\draw (-2,-0.3) node {\footnotesize{$\mc c_\infty(y_0\!)$}};\end{tikzpicture}
\begin{tikzpicture}[scale=0.65]
\draw (-2.8,5.1) node {b)};
\draw[thin,->] (-2.1,0)--(3,0);
\draw[thin,->] (-2,-0.1)--(-2,5);
\draw[thin,-] (-2,0)--(3,5);
\draw[-,thick] (0.05,2.05)--(1,3);
\draw[scale=1, domain=-2:0, smooth, variable=\x, black,thick] plot ({\x}, {0.5*(\x+2))});
\draw  [thick] (1,3)..controls (1.5,1.8) and (2.5,1.8)..(3,4);
\draw (3,-0.6) node {\small{$y$}};
\draw (-2.25,5) node {\small{$x$}};
\draw (-0.8,4) node {\footnotesize{$b(x,y)=0$}};
\draw (2.15,4.8) node {\footnotesize{$x=y$}};
\draw (3.1,4.1) node {\footnotesize{$x=\mc b(y)$}};
\draw (0,1) circle  [radius=0.05];
\fill (0,1) circle  [radius=0.05];
\draw (0,2) circle  [radius=0.05];
\fill (1,3) circle  [radius=0.05];
\draw [thin,dotted] (1,3)--(1,0);
\draw [thin,dotted] (-2,2.19)--(1.85,2.19);
\draw [thin,dotted] (1,3)--(-2,3);
\draw (-3,3.1) node {\footnotesize{$\mc b(y_0\!)=y_0$}};
\draw (-3.7,2.1) node {\footnotesize{$\begin{array}{cc}\mc c(y_0\!)\!\!\!\!&\!\!=\!\mc c^{(2)}(y_0)\\&\!\!= {\mc c}_\infty(y_0)\end{array}$}};
\draw (1,-0.2) node {\footnotesize{$y_0$}};
\end{tikzpicture}
\caption{The case when the supp\,$b$ allows for jumping over the gap where $\mc b(y)= y$ (figure a))
and where the jump is too short (figure b)). }\label{fig1b}
\end{center}
\end{figure}
{\begin{proof}
By \eqref{RESOL}, we have
\begin{align*}
[R(\lambda,Z_{\beta,m})]f(x)&=[R(\lambda,Z_{0,m})]f(x)
+ e_\la(x)\frac{\langle\beta ,R(\lambda,Z_{0,m})f\rangle}{1-\langle\beta ,e_\la\rangle}\\
&=[(I+E_\la)R(\lambda,Z_{0,m})]f(x),\qquad f\in X_m,
\end{align*}
where
$$
[R(\la,Z_{0,m})f](x) = \frac{e^{-\lambda R(x) -Q(x)}}{r(x)}\int _0^xf(y)e^{\lambda R(y) +Q(y)}dy
$$
and
\begin{equation}
R(\la, K_{\beta,m})=R(\la,Z_{\beta,m})\sum_{n=0}^{\infty }\left[ B_mR(\la,Z_{\beta,m})\right] ^{n}.\label{res9}
\end{equation}
Let $g > 0$ and set $z_g = \sup\{z\;: g(z) = 0\; a.e.\; \text{on}\; [0,z]\}$. If $z_g =0$,
then already $R(\la,Z_{\beta,m})g\geq R(\la,Z_{0,m})g >0$, and the result is valid. Assume then that $z_g >0$
and observe that
$$
\Psi_0(z) := [R(\la,Z_{0,m})g](z) = \left\{\begin{array}{lcl}\frac{e^{-\lambda R(z) -Q(z)}}{r(z)}
\int _{z_g}^zg(y)e^{\lambda R(y) +Q(y)}dy,&\text{for}& z> z_g,\\
0,&\text{for}& 0\leq z \leq z_g.\end{array}\right.
$$
Hence, $[R(\la,Z_{0,m})g](z)>0$ is strictly positive for $z>z_g$, while
$$
\langle\beta,\Psi_0\rangle = \int_{z_g}^{\infty}\beta(x)\Psi_0(x)dx>0
$$
and
$$
[R(\la,K_{\beta,m})g](z)\geq [R(\la,Z_{\beta,m})g](z)\geq
\frac{e_\la(z)}{1-\langle\beta ,e_\la\rangle}\langle\beta,\Psi_0\rangle>0,
$$
provided \eqref{betas} is satisfied.
Hence, $R(\la,K_{\beta,m})g$ is positivity improving.

If \eqref{betas} is not satisfied, let us consider the terms $R(\la,Z_{\beta,m})\left[ B_m R(\la,Z_{\beta,m})\right]^{n}$
of \eqref{res9}. Denoting $\bar e_\la = e_\la/(1-\langle\beta ,e_\la\rangle)$, we see that
\begin{align*}
&R(\la,Z_{\beta,m}) B_mR(\la,Z_{\beta,m}) = R(\la,Z_{0,m}) B_mR(\la,Z_{0,m})
+ R(\la,Z_{0,m}) B_m\bar e_{\la} \langle\beta, R(\la,Z_{0,m})\cdot\rangle\\
&\qquad\qquad\qquad
+ \bar e_\la \langle \beta,R(\la,Z_{0,m}) B_mR(\la,Z_{0,m})\cdot\rangle
+ \bar e_\la \langle \beta, R(\la,Z_{0,m})B_m\bar e_\la\rangle \langle \beta, R(\la,Z_{0,m})\cdot\rangle\\
&\qquad\qquad\qquad
=R(\la,Z_{0,m}) B_mR(\la,Z_{0,m}) + \sum\limits_{i=0}^1 f^1_i
\langle\beta, R(\la,Z_{0,m}) [B_mR(\la,Z_{0,m})]^i\cdot\rangle,
\end{align*}
where $f^1_i$ are scalar functions. Hence, making the inductive assumption
\begin{equation}\label{indass}
\begin{split}
R(\la,Z_{\beta,m}) [B_mR(\la,Z_{\beta,m})]^{n}&=R(\la,Z_{0,m}) [B_mR(\la,Z_{0,m})]^{n} \\
&+ \sum\limits_{i=0}^{n} f^{n}_i \langle\beta, R(\la,Z_{0,m}) [B_mR(\la,Z_{0,m})]^i\cdot\rangle
\end{split}
\end{equation}
for some functions $f^n_i, i=0,\ldots,n$, we have
\begin{align*}
&R(\la,Z_{\beta,m}) [B_m R(\la,Z_{\beta,m})]^{n+1}\\
&\qquad
= R(\la,Z_{\beta,m}) [B_mR(\la,Z_{\beta,m})]^{n}(B_mR(\la,Z_{0,m})
+ B_m\bar e_\la \langle\beta,R(\la,Z_{0,m})\cdot\rangle\\
&\qquad
= R(\la,Z_{0,m}) [B_mR(\la,Z_{0,m})]^{n+1}
+ R(\la,Z_{0,m}) [B_mR(\la,Z_{0,m})]^{n}B_m\bar e_\la \langle\beta,R(\la,Z_{0,m})\cdot\rangle\\
&\qquad
+ \sum\limits_{i=0}^n f^n_i \langle\beta, R(\la,Z_{0,m}) [B_mR(\la,Z_{0,m})]^{i+1}\cdot\rangle \\
&\qquad
+ \sum\limits_{i=0}^n f^n_i \langle\beta, R(\la,Z_{0,m}) [B_mR(\la,Z_{0,m})]^iB_m\bar e_\la
\rangle \langle\beta,R(\la,Z_{0,m})\cdot\rangle\\
&\qquad
=R(\la,Z_{0,m}) [B_mR(\la,Z_{0,m})]^{n+1}
+ \sum\limits_{i=0}^{n+1} f^{n+1}_i \langle\beta, R(\la,Z_{0,m}) [B_mR(\la,Z_{0,m})]^i\cdot\rangle,
\end{align*}
hence \eqref{indass} is proved. We also obtain $f_i^{n+1} = f_{i-1}^n$ for $i>1$ and
$$
f_{0}^{n+1} =  R(\la,Z_{0,m}) [B_mR(\la,Z_{0,m})]^{n}B_m\bar e_\la
+ \sum\limits_{i=0}^n f^n_i \langle\beta, R(\la,Z_{0,m}) [B_mR(\la,Z_{0,m})]^iB_m\bar e_\la \rangle
$$
for $n\geq 1$, with $f^0_0 = \bar e_\la$.

Next, for $z<z_g$,
\begin{align*}
&\Psi_1(z):=[R(\la,Z_{0,m})B_m R(\la,Z_{0,m})g](z) = [R(\la,Z_{\beta,m})B_m \Psi_0](z)\\
&=\frac{e^{-\lambda R(z) -Q(z)}}{r(z)}\int_{0}^{z}e^{\lambda R(y) +Q(y)}
\left(\int_{z_g}^{\infty}a(s)b(y,s)\Psi_0(s)ds\right) dy,
\end{align*}
where the inner integration is, in fact, carried out over $\mc A\cap [z_g,\infty)$. Hence,
$\Psi_1(z)>0$ for $z>\mc c(z_g)$. So, in particular, if $a(y)b(x,y)>0$ for all $y>0$ and $0< x <y$,
so that $\mc b(y)=0$ for any $y>0$, then the result is proved. If not, then for the third term, we have
$$
\Psi_2(z):=[R(\la,Z_{0,m})[B_m R(\la,Z_{0,m})]^2 g](z) = [R(\la,Z_{0,m})B_m \Psi_1](z)
$$
and thus, by the same argument, $\Psi_2(z)>0$ for $z> \mc c^{(2)}(z_g)$. Using induction, we conclude that
$[R(\la,Z_{\beta,m})g](z)>0$ almost everywhere on $(\mc c_\infty(z_g), \infty)$.

Hence, if $\bar{\mc c} =0$, the theorem is proved. If not, by $\sup {\rm supp}\,\beta > \bar{\mc c}$ we obtain
$$
\langle\beta, R(\la,Z_{0,m}) [B_mR(\la,Z_{0,m})]^{n} g\rangle >0,
$$
for sufficiently large $n$ and, since $f_n^n(z) = \bar e_\la (z)>0$ for all $z>0,$  the result follows.

To prove the necessity, negating \eqref{betas}--\eqref{ass2}, we have $0<\bar{\mc c}<\infty$ and
$\sup{\rm supp}\,\beta \leq \bar{\mc c}$. Let us consider $f$ supported in $[z,\infty)$ for some $z>\bar{\mc c}$.
Then, by Lemma \ref{lem2}, the supports of $[B_m R(\la,Z_{0,m})]^{n} f$, $n\ge 0$, are confined to
$[\bar{\mc c},\infty)$ and thus
$$
\langle\beta, R(\la,Z_{0,m}) [B_mR(\la,Z_{0,m})]^{n} g\rangle = 0,
$$
for any $n$ due to the assumption on $\beta$. Hence ${\rm supp}\, R(\la, K_{\beta,m})f \subset [\bar{\mc c},\infty)$ and the semigroup is not positivity improving.
\end{proof}
This theorem generalizes many earlier results.
\begin{remark} \label{exMKJB} (a) For instance, in \cite{MKJB} the authors used the assumption that for any
$x>0,$ $\mathrm {supp}\; a(\cdot)b(x, \cdot)$ is infinite, see Fig. \ref{fig1a}. In our notation, it means that for any
$x>0$ and any $z$ there is $y>z$ such that
$\mc b(y) \leq x$. This implies that for any $x$ and any $z$, $\mc c(z)\leq x$, which yields ${\mc c}_\infty (z)\leq x$
and hence $\bar{\mc c}\leq x$. Since $x$ is arbitrary, $\bar{\mc c}=0$.

(b) If $a \in C(\mbb R_+)$, then the conditions of Theorem \ref{thirred} can be made more explicit. Indeed, by
\cite[Theorem 3.10]{Apo}, there exist increasing sequences $(\alpha_k)_{k\in \mbb Z}$ and $(\beta_k)_{k\in \mbb Z}$,
with $0<\alpha_k<\beta_k<\alpha_{k+1}<\infty$, such that
$$
{\rm supp}\, a = \bigcup\limits_{k\in \mbb Z} (\alpha_k,\beta_k),
$$
so that the null-set is given by
$$
\mbb R_+\setminus {\rm supp}\, a = \bigcup\limits_{k\in \mbb Z} [\beta_k,\alpha_{k+1}].
$$
Then $\bar{\mc c} = 0$ if and only if for each $k\in \mbb Z$ there exists $r>k$ such that $\mc b(y)<\beta_k$
for some $y \in (\alpha_r,\beta_r)$. In other words, for any gap  $[\beta_k,\alpha_{k+1}]$ of the support of $a$, there must
be a point in one of the following intervals of the support of $a$ on which ${\rm supp}\, b$ overlaps with $[\beta_k,\alpha_{k+1}]$.

(c) On physical grounds, we expect that if a particle fragments, it produces at least two daughter particles,
and not all daughter particles will have masses close to $y$, see \cite[Section 2.2.3.2]{JWL}. As shown in
\cite[Theorem 5.2.21]{JWL}, for irreducibility it is even sufficient to have $n_0(y)\geq 1+\delta$
for some $\delta>0$ independent of $y$ (that is, $\mc b(y)\leq 1+\delta$ for all $y\in \mbb R_+$). We can generalize this result by showing
that if there is a continuous function $\phi$ such that
$$
\mc c(y) \leq \phi(y) <y,
$$
for any $y>0$, then $\bar{\mc c}=0$, and thus \sem{S_{K_{0,m}}} (and hence \sem{S_{K_{\beta,m}}}
for any $\beta\geq 0$) are irreducible. Indeed, consider any $z>0$ and the corresponding $c_\infty(z)$. Then, from
the monotonicity of $\mc c$, we obtain
$$
\mc c^{(2)}(z) = \mc c(\mc c(z)) \leq \mc c(\phi (z)) \leq \phi^{(2)}(z)
$$
and hence $c_\infty(z)\leq \phi_\infty(z)\leq z$. If $c_\infty(z) \neq 0$, then also $\phi_\infty(z) \neq 0$.
But $\phi_\infty(z)$ is an equilibrium of the discrete dynamical system $r_{n+1} = \phi(r_n)$, $r_0 = z$, i.e.,
$\phi(\phi_\infty(z)) = \phi_\infty(z)$, which is impossible by our assumption on $\phi$.
\end{remark}

\section{The spectral gap}\label{subsecsg}
The analysis presented in Sections~\ref{secGT} and \ref{secsg}
paves the way for a complete description of the large time asymptotics of \eqref{Tsapele}. Indeed, as observed in
\cite[Corollary~22]{MKJB}, under assumptions \eqref{RRe}, \eqref{g1}, \eqref{n0}, \eqref{cm}, \eqref{mass}
and the additional hypothesis \eqref{Tss} (see the discussion in Section~\ref{secMKJB}),
the semigroup $(S_{K_{0,m}}(t))_{t\ge 0}$ is resolvent compact. Combining this fact with Theorems~\ref{thpert} and
\ref{thirred}, we have

\begin{theorem}\label{thspgap}
Assume  \eqref{RRe}, \eqref{g1}, \eqref{aa1}, \eqref{n0}, \eqref{cm}, \eqref{Tss}, \eqref{massum} and \eqref{betaas}
are satisfied. Assume further that one of the hypotheses of Theorem~\ref{thirred} holds.
Then $(S_{K_{\beta,m}}(t))_{t\ge0}$ has a spectral gap, i.e.,
\begin{equation}\label{spgap}
r_{ess}\bigl(S_{K_{\beta,m}}(t)\bigr) <r_\sigma \bigl(S_{K_{\beta,m}}(t)\bigr),\quad t>0,
\end{equation}
where $r_{ess}(\cdot)$ and $r_\sigma(\cdot)$ denote the essential spectral and the spectral radii, respectively.
\end{theorem}
\begin{proof}
(a) The first group of assumptions ensures that the resolvents $R(\lambda, Z_{0,m})$ and $R(\lambda, K_{0,m})$
are compact. These fact, combined with Theorem~\ref{resG} and \eqref{reseq}
(see the proof of Theorem~\ref{LIN2}),
indicates that $R(\lambda, Z_{\beta,m})$ and $R(\lambda, K_{\beta,m})$, with $\beta>0$, are compact, provided
$\lambda$ is sufficiently large. In addition, from Theorem~\ref{resG} and \eqref{reseq}
it follows that
\[
R(\lambda, K_{0,m})\le R(\lambda, K_{\beta,m}).
\]

(b) The second group of assumptions shows that both $R(\lambda, K_{0,m})$ and $R(\lambda, K_{\beta,m})$ are irreducible.
Since $R(\lambda, K_{0,m})\ne R(\lambda, K_{\beta,m})$, when $\beta>0$, the comparison theory of
\cite[Theorem~4.3]{IVo}, combined with the spectral theory of positive $C_0$-semigroups
(see e.g. \cite[Chapter VI]{Klaus1}), indicates that
\[
\frac{1}{\lambda - s(K_{0,m})} = r_{\sigma}(R(\lambda, K_{0,m})) < r_{\sigma}(R(\lambda, K_{\beta,m}))
= \frac{1}{\lambda - s(K_{\beta,m})},
\]
so that
\[
r_{\sigma}\bigl(S_{K_{0,m}}(t)\bigr) = e^{s(K_{0,m}) t}
< e^{s(K_{\beta,m}) t} = r_{\sigma}\bigl(S_{K_{\beta,m}}(t)\bigr),\quad t>0.
\]
On the other hand, from \eqref{betaas} (see Theorem~\ref{thpert} and \eqref{eqsim}), we have
\[
r_{ess}\bigl(S_{K_{\beta,m}}(t)\bigr) = r_{ess}\bigl(S_{K_{0,m}}(t)\bigr) \le r_{\sigma}\bigl(S_{K_{0,m}}(t)\bigr).
\]
Combining the last two inequalities, we arrive at \eqref{spgap} and the claim is settled.
\end{proof}

In the setting of Theorem~\ref{thspgap}, the resolvent $R(\lambda, S_{K_{\beta,m}})$ falls into the scope of the
Perron-Frobenius theory for positive irreducible compact operators. As a byproduct of this theory, we see that
$s(K_{\beta,m})$ is the simple dominant isolated eigenvalue of $(K_{\beta,m}, D(Z_{\beta,m}))$.
Let $u_{\beta,m}\in X_{m,+}$ and $u_{\beta,m}^\ast \in X_{m,+}^\ast$ be the associated positive eigenfunctions
of  $(K_{\beta,m}, D(Z_{\beta,m}))$ and its transpose $(K_{\beta,m}^\ast, D(K_{\beta,m}^\ast))$, respectively.
In the usual way, we define the eigenprojector
\[
\mathcal{P} = u_{\beta,m} \langle u_{\beta,m}^\ast, \cdot\rangle.
\]
With this notation, the standard asynchronous exponential growth theory for positive $C_0$-semigroups
(see e.g. \cite[Theorem~VI.3.5]{Klaus1}) yields
\begin{corollary}\label{asyexpgr}
Under the hypothesis of Theorem~\ref{thspgap}, there exist positive constants $\varepsilon$ and $c$, such that
\begin{equation}\label{eqasyexpgr}
\|e^{-s(K_{\beta,m}) t} S_{K_{\beta,m}}(t)f - \mathcal{P}f \|_m \le c e^{-\varepsilon t} \|f\|_m,\quad t\ge 0,
\end{equation}
for all $f\in X_m$.
\end{corollary}
\section{Explicit solutions to the growth-fragmentation problem with McKendrick-von
Foerster boundary conditions}\label{seces}

To illustrate the theory presented above, in this section we focus on a special case of \eqref{Tsapele} with the coefficients
\begin{equation}\label{binfrag}
r(x)=r>0,\quad a (x)=ax,\; a>0,\quad b(x,y) = \frac{2}{y}\;\;\text{and}\;\; \beta (x)=\beta _0 +\beta _1x,\; \beta_0,\beta_1\geq 0.
\end{equation}
In this case, the theory
of \cite{Banasiak2022} yields closed form solutions and the abstract calculations of Sections~\ref{secGT}--\ref{secsg}
can be made more explicit.

\subsection{Closed form solutions.} Under assumptions \eqref{binfrag}, \eqref{Tsapele} takes the form
\begin{subequations}\label{Main1}
\begin{align}
\label{Ma}
&\p_t u(x,t)=-r\p_x u(x,t)-a xu(x,t) +2a\int _x^{\infty}u(y,t)dy,\quad x,t>0,\\
\label{Mb}
&u(x,0)=u_0(x),\quad x>0,\\
\label{Mc}
&u(0,t)=\beta _0M_0(t) +\beta _1M_1(t),\quad t>0,
\end{align}
\end{subequations}
where the moments of the solution are defined by $M_k(t)=\int _0^{\infty}x^ku(x)dx$, $k\in\mathbb{N}$.
We see that $l=0$ in \eqref{n0} and $p=1$ in \eqref{aa1}, $\beta \in X_m^*$ for any $m>1$ and \eqref{lim}
is satisfied. Hence, there exists a quasi-contractive positive semigroup $(S_{K_{\beta,m}}(t))_{t\geq 0}$
solving \eqref{Main1}, such that $u(\cdot,t) = [S_{K_{\beta,m}}(t)u_0](\cdot)\in D(Z_{\beta,m})$,
whenever $u_0\in D(Z_{\beta,m})$. As a consequence, for regular input data $u_0$ the semigroup solutions
are strongly differentiable in $X_m$ and
\[
\int _0^{\infty}x^i \p_tu(x,t)dx=\dfrac{dM_i(t)}{dt}, \quad i=0,1.
\]
On the account of Lemma \ref{est10} and \eqref{Mc}, the first two moments satisfy
\begin{equation}\label{moments}
\begin{split}
M'_0(t)&=\alpha_0M_0(t) + \alpha_1M_1(t),\\
M'_1(t)&=rM_0(t),
\end{split}
\end{equation}
where for the sake of brevity, we let $\alpha _0 :=r\beta _0$ and $\alpha _1 :=r\beta _1+a$.
The eigenvalues of the matrix, appearing in the right-hand side of \eqref{moments}, are
\begin{equation}\label{lambdaeq}
\lambda _{\pm}=\dfrac{\alpha _0 \pm \sqrt{\alpha _0^2 +4r\alpha _1}}{2}
\end{equation}
and the solution are given explicitly by the formulae
\begin{equation}
\begin{aligned}
M_0(t)&=
\frac{\lambda_+e^{\lambda _{+}t}-\lambda_-e^{\lambda _{-}t}}{\lambda_+-\lambda_-} M_0(0)
+\frac{\lambda_+\lambda_-(e^{\lambda _{+}t}-e^{\lambda _{-}t})}{r(\lambda_- -\lambda_+)} M_1(0) \\
&=K_{0,0}(t) M_0(0) +K_{0,1}(t) M_1(0),\\
M_1(t)&= \frac{r(e^{\lambda_+ t} - e^{\lambda_- t})}{\lambda_+-\lambda_-}M_0(0)
+\frac{\lambda_+ e^{\lambda _{-}t}-\lambda_-e^{\lambda _{+}t}}{\lambda_+-\lambda_-}M_1(0)\\
&=K_{1,0}(t)M_0(0)+K_{1,1}(t)M_1(0).
\end{aligned}\label{mom}
\end{equation}
From the inequality $\lambda_-<0<\lambda_+$, it follows that $K_{i,j}(t)\ge 0$, $i,j=0,1$, for $t\ge 0$.
Hence, $M_i(t)$, $i=0,1$, are positive for positive regular data $u_0$ from $D(Z_{\beta,m})_+$.

\begin{figure}
\begin{center}
\includegraphics{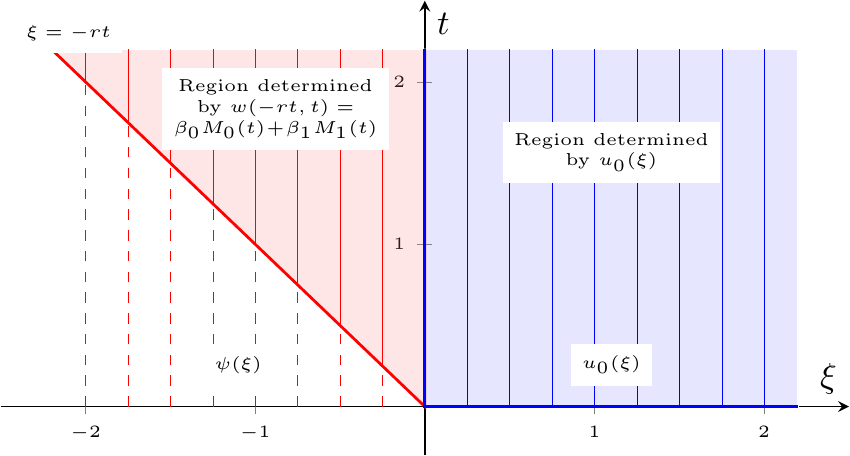}
 \caption{Geometry of the problem \eqref{treq1}--\eqref{treq2} in the characteristic coordinates $(\xi,t)$:
$u_0(\xi)\ne 0$ for $\xi>0$; $\psi(\xi)$ is to be determined for $\xi<0$ so that
$w(-rt,t) = \beta _0M_0(t) +\beta _1M_1(t)$, $t>0$.}\label{fig1}
\end{center}
\end{figure}

As mentioned earlier, \eqref{Main1} falls into the scope of the theory presented in \cite[Section 5.3]{Banasiak2022},
with the only difference being that the boundary condition \eqref{Mc} is no longer homogeneous.
To obtain closed formulas as in \cite{Banasiak2022}, we pass to the characteristic coordinates, i.e., we let
\[
x = rt+\xi,\quad v(\xi,t) = u(rt+\xi,t).
\]
Direct substitution transforms \eqref{Ma}, \eqref{Mb} to
\begin{subequations}\label{treq}
\begin{equation}\label{treq1}
\begin{split}
\p_tv(\xi,t)&=-a(rt+\xi)v(\xi,t) + 2a\int _{\xi}^{\infty}v(s,t)ds,\quad rt> \max\{0,-\xi\},\\
v(\xi ,0)&=u_0(\xi),\quad \xi>0.
\end{split}
\end{equation}
The characteristic transformation maps the entire first quadrant of the
$(x,t)$-plane onto the region $S = \bigl\{(\xi,t)\,|\, rt\ge \max\{0,-\xi\}\bigr\}$ of the $(\xi,t)$-plane,
however, the family of characteristic lines starting in the positive part of $\xi$-axis do not cover the entire
set $S$. As a consequence, in these coordinates the initial data $u_0(\xi)$
determines the solution only for $\xi>0$, see Fig.~\ref{fig1}.
To determine the solution in the sector $-rt <\xi <0$, we solve \eqref{treq1} in the whole upper $(\xi,t)$-plane.
For that, we extend the original initial data $u_0$ by letting $u_0(\xi)=0$ for $\xi < 0$ and define new
initial condition $\phi (\xi)=u_0(\xi)+\psi({\xi})$ so that $\psi (\xi) =0$ for $\xi>0$,
while the associated extended solution $v$ satisfies
\begin{equation}\label{treq2}
v(-rt, t)= \beta _0M_0(t) +\beta _1M_1(t).
\end{equation}
\end{subequations}

As in \cite{Banasiak2022}, we set $w(\xi,t):=e^{at\big(\xi +\tfrac{rt}{2}\big)}v(\xi,t)$,
so that \eqref{treq1} (extended to $\xi<0$) takes the form
\begin{equation}
\begin{aligned}
\p_t w(\xi ,t)&=2a\int _{\xi}^{\infty}e^{-at(s-\xi)}w(s,t)ds,\quad t>0, \quad \xi\in\mathbb{R},\\
w(\xi, 0)&=u_0(\xi) +\psi (\xi),\quad \xi \in \mathbb{R},
\end{aligned}
\end{equation}
where $u_0(\xi) =0$ for $\xi<0$, while $\psi (\xi)=0$ for $\xi>0$. By \cite[Section 5.3]{Banasiak2022}, we have
\begin{subequations}\label{soln}
\begin{align}
\label{solna}
w(\xi ,t)&=[(I+at\mathcal{J}^+)^2u_0](\xi),\quad \xi>0,\\
\nonumber
w(\xi ,t)&=[(I+at\mathcal{J}^+)^2u_0](\xi)+[(I+at\mathcal{J}^+)^2\psi](\xi) \\
\label{solnb}
&= [(I+at\mathcal{J}^+)^2u_0](\xi)+[(I+at\mathcal{J})^2\psi](\xi),\quad \xi<0,
\end{align}
\end{subequations}
where $I$ is the identity operator, $\mc J^+$ denotes the antiderivative
$[\mathcal{J}^+f](\xi) = \int_{\xi}^{\infty}f(s)ds$, $\xi\in \mbb R$ and
$[\mathcal{J}f](\xi):=\int _{\xi}^0f(s)ds$, $\xi<0$. From \eqref{soln}, it follows that
\begin{subequations}\label{eq71}
\begin{align}
\label{eq71a}
w(\xi ,t) &= u_0(\xi) + at\int_\xi^\infty \bigl[2 + at (s-\xi)\bigr] u_0(s) ds,\quad \xi>0,\\
\label{eq71b}
w(\xi ,t) & = at(2-at\xi)M_0(0) + a^2t^2 M_1(0) +
\psi(\xi) + at\int_\xi^0 \bigl[2 + at (s-\xi)\bigr] \psi(s) ds,\quad \xi<0.
\end{align}
\end{subequations}
On the account of \eqref{eq71b}, \eqref{treq2} and our definition of $w(\xi,t)$,
we conclude that the unknown initial data $\psi(\xi)$ satisfies
\begin{subequations}\label{dj}
\begin{equation}\label{dja}
\psi(-rt) + 2at  \int _{-rt}^0\psi (s)ds + a^2t^2 \int _{-rt}^0(s+rt)\psi (s)ds
= F(t),\quad t>0,
\end{equation}
with
\begin{equation}\label{djb}
F(t) = e^{-\frac{art^2}{2}}\bigl[\beta_0M_0(t) + \beta_1M_1(t)\bigr]
- \bigl[2at + ra^2t^3\bigr] M_0(0) - a^2t^2M_1(0),\quad t>0.
\end{equation}
\end{subequations}
As in \cite{Banasiak2022}, we let $Y(t)=e^{\frac{art^2}{2}}\int _{-rt}^0(s+rt)\psi (s)ds$.
Then \eqref{dj} takes the form
\[
Y''(t) - ar Y(t) = r^2e^{\frac{art^2}{2}} F(t), \quad Y(0)=Y^{\prime}(0)=0.
\]
The standard variation of constant formula, together with the homogeneous initial data, yields the solution
\[
Y(t) = \frac{r^2}{\sqrt{ar}}\int_0^t \sinh\bigl[\sqrt{ar}s\bigr] e^{\frac{ar(t-s)^2}{2}} F(t-s) ds,\quad t>0,
\]
which, upon backward substitution and differentiation, gives
\begin{equation}\label{eq73}
\begin{split}
\psi(\xi) &= \Bigl[\cosh\Bigl(\sqrt{\frac{a}{r}}\xi\Bigr) - 2\sqrt{\frac{a}{r}}\sinh\Bigl(\sqrt{\frac{a}{r}}\xi\Bigr)\Bigr]
 e^{-\frac{a\xi^2}{2r}} F(0)
- \frac{1}{\sqrt{ar}}\sinh\Bigl(\sqrt{\frac{a}{r}}\xi\Bigr)
 e^{-\frac{a\xi^2}{2r}} F'(0)\\
&-\sqrt{\frac{r}{a}}\int_{\xi}^0
\sinh\Bigl(\sqrt{\frac{a}{r}}s\Bigr) \frac{d^2}{d\xi^2}\Bigl[e^{\frac{as(s-2\xi)}{2r}} F\Bigl(\frac{s-\xi}{r}\Bigr)\Bigr]ds,
\quad \xi<0.
\end{split}
\end{equation}
Returning to the original variables $x$ and $t$ in \eqref{eq71} and \eqref{eq73}, we finally arrive at
the explicit formula
\begin{equation}\label{exsol}
\begin{aligned}
u(x,t)=e^{\frac{at(rt-2x)}{2}}
\begin{cases}
\displaystyle
u_0(x -rt) +at\int\limits_{x-rt}^\infty \bigl[ 2 +at(s-x+rt)\bigr] u_0(s)ds, & \text{if }  x \geq rt;\\
\displaystyle
at(2+art^2-atx)M_0(0) + a^2t^2 M_1(0) & \\
\displaystyle
\quad\quad
+\psi(x -rt) +at\int\limits_{x-rt}^0 \bigl[ 2 + at(s-x+rt)\bigr] \psi(s)ds, & \text{if }  0\leq x < rt.
\end{cases}
\end{aligned}
\end{equation}
Direct calculations show that \eqref{exsol} is indeed a $C_0$-semigroup solutions, i.e., for $u_0\in X_m$ and any
$m>1$, we have $\lim_{t\to 0^+} u(t) = u_0$ in $X_m$, $u(t)\in X_m$ and if $u(t;u(s))$ is given
by \eqref{exsol}, with initial data $u(s)$ instead of $u_0$, $u(t;u(s)) = u(t+s)$.

\subsection{Large time asymptotic.}
To begin with, we note that \eqref{exsol} agrees with the theory of Section~\ref{subsecsg}.
Indeed, writing
\[
u^-(x,t) = \chi_{[0,rt)}(x) u(x,t),\quad u^+(x,t) = \chi_{[rt,+\infty)}(x) u(x,t),
\]
it is not difficult to verify that the contribution of $u^+(x,t)$ to the large time asymptotic of the solution is negligible.
On the account of \eqref{exsol}, for $u_0\in X_m$, $m>1$,  we have
\begin{align*}
\|u^+(t)\|_m  &= e^{\frac{-art^2}{2}} \int_0^\infty (1+(x+tr)^m) e^{-ats} u_0(s)ds\\
&+ at e^{\frac{-art^2}{2}} \int_0^\infty u_0(s)ds \int_0^s (1+(x+rt)^m) \bigl[ 2 +at(s-x)\bigr] e^{-atx} dx.
\end{align*}
Partial integration, together  with the elementary inequality $(x+y)^m\le 2^m(x^m+y^m)$, $x,y\ge 0$,
$m\ge 1$, shows that the bounds
\begin{align*}
&(1+(x+rt)^m) \le c_0t^m(1+s^m),\\
&at \int_0^s (1+(x+rt)^m) \bigl[ 2 +at(s-x)\bigr] e^{-atx} dx \le c_1 t^{m+1}(1+s^m),
\end{align*}
hold uniformly for large values of $t>0$, with some $c_0,c_1>0$ that depend on $m$, $a$ and $r$ only.
Hence,
\begin{equation}\label{asyplus}
\|u^+(\cdot, t)\|_m  \le c t^{m+1}e^{\frac{-art^2}{2}} \|u_0\|_m,
\end{equation}
with some $c>0$ and the bulk asymptotics of $u(x,t)$ is governed by $u^-(x,t)$. According to \eqref{exsol}, and
in agreement with Corollary~\ref{asyexpgr}, $u^-(x,t)$ is solely determined by the first two moments $M_i(t)$, $i=0,1$.

To obtain explicit formula for the principal asymptotic term, we employ the theory of Sections~\ref{secsg}--\ref{subsecsg}.
By virtue of Theorem~\ref{thirred} and assumptions \eqref{binfrag}, $s_0:=s(K_{\beta,m})$
is a positive simple dominant
eigenvalue of $(K_{\beta,m}, D(Z_{\beta,m}))$, $m>1,$ and hence it satisfies
\begin{equation}\label{eigprob1}
s_0 v  - K_{\beta,m}v = 0,\quad v\in D(Z_{\beta,m})_+.
\end{equation}
In our setting,  the eigenvalue problem \eqref{eigprob1} is
\begin{subequations}\label{eigprob2}
\begin{align}
\label{eigprob2a}
&s_0 v (x) +rv'(x) + axv(x) - 2a\int_x^\infty v(y)dy = 0,\quad x>0,\\
\label{eigprob2b}
&v(0) = \beta_0 M_0 + \beta_1 M_1,\quad  v,v_x\in X_m,\; m>1.
\end{align}
\end{subequations}
Differentiating \eqref{eigprob2a}, we obtain the following second order linear homogeneous ODE
\begin{equation}\label{eigprob3}
rv''(x) +  (s_0 +ax) v' (x) + 3av(x) = 0.
\end{equation}
Thanks the special relation between the coefficients,  \eqref{eigprob3} is integrable. Indeed, upon the change of variables $z= (s_0+ax)$, the equation takes the form
\begin{equation}\label{eigprob3a}
\ti v''(z) +  \frac{z}{ar} \ti v' (z) + \frac{3}{ar}\ti v(z) = 0,
\end{equation}
which is of the form of \cite[Section 2.1.2.20]{PZ2003}, a particular solution of which is given by
$$
\ti v(z) = \Phi \left(\frac{3}{2},\frac{1}{2}; - \frac{z^2}{2ar}\right) = e^{-\frac{z^2}{2ar}}\Phi \left(-1,\frac{1}{2};  \frac{z^2}{2ar}\right).
$$
Here, $\Phi$ is Kummer's function, see \cite[Formula 13.1.2]{abr} (though with symbol $M$ instead of $\Phi$) and we used \cite[ Formula 13.1.27]{abr} for the transformation. Then, using \cite[Formula 13.6.17]{abr} (or by direct substitution)
$$
\Phi \left(-1,\frac{1}{2};  \frac{z^2}{2ar}\right)= - He_{2}\left(\frac{z}{\sqrt{ar}}\right) = - \frac{1}{2}H_{2}\left(\frac{z}{\sqrt{2ar}}\right) = 1-\frac{z^2}{ar},
$$
where $He_{2}$ and $H_2$ are, respectively, the probabilistic and physicist's Hermite polynomials of second order, see \cite[Chapter 22]{abr}. Taking into account that the second solution to \eqref{eigprob3a} can be found by the formula
\begin{equation}
\ti u (z) = \ti v(z)\int \frac{e^{-\frac{z^2}{2ar}}}{\ti v^2(z)} dz,
\label{secsol}
\end{equation}
which satisfies $\ti u (z) = O(z)$ as $z\to \infty$, we see that the only (up to a multiplicative constant) $X_m$-solution, $m>1$, is given by
\[
v(x) = e^{-\frac{(ax+s_0)^2}{2ar}} \left(\frac{(ax+s_0)^2}{ar}  - 1\right),\quad x>0.
\]
The solution is nonnegative if $s_0>\sqrt{ar}$ and satisfies \eqref{eigprob2b} if and only if
\begin{equation}\label{chareq}
s_0^2 - \alpha_0s_0 - r\alpha_1 = 0.
\end{equation}
The only positive root of \eqref{chareq} is given by $\lambda_+$, where $\lambda_+$ is defined in \eqref{lambdaeq}.
Since $\lambda+>\sqrt{ar}$, it follows that
\begin{equation}\label{eigpair1}
(s_0, v_0(x)) :=
\left(\lambda_+, \kappa e^{-\frac{(ax+s_0)^2}{2ar}} \left(\frac{(ax+s_0)^2}{ar}  - 1\right) \right),\quad
\kappa = \frac{a}{\lambda_+} e^{\frac{\lambda+^2}{2ar}},
\end{equation}
is indeed the simple dominant positive eigenpair of $(K_{\beta,m}, D(Z_{\beta,m}))$, for any $m>1$.
Our choice of the normalization constant $\kappa$ in \eqref{eigpair1} guarantees that $\|v_0\|_0 = 1$.

To determine the associated eigenprojector $\mathcal{P}$, we consider the eigenvalue problem for
the transpose $(K_{\beta,m}^\ast, D(K_{\beta,m}^\ast))$ of $(K_{\beta,m}, D(K_{\beta,m}))$, in
$X_m^\ast$, $m>1$. Direct calculations show that the eigenvalue problem in $X_m^\ast$ is
\begin{subequations}\label{eigprob1p}
\begin{align}
\label{eigprob1pa}
&\lambda w (x)  - rw''(x) + axw(x) - 2a\int_0^x w(y)dy - r(\beta_0+\beta_1 x)\lim_{x\to0^+} w(x) = 0,\quad x>0,\\
\label{eigprob1pb}
&\text{there exists $\lim_{x\to0^+} w(x)$},\quad w,w'\in X_m^\ast,\; m>1.
\end{align}
\end{subequations}
As before, we differentiate \eqref{eigprob1pa} with respect to $x$ to get
\begin{equation}
-rw''(x) + (\la +ax)w'(x) - a w(x) -r\beta_1 w(0) = 0.
\label{ep1}
\end{equation}
Let $$[L_\la f](x) =  -rf''(x) + (\la +ax)f'(x) - a f(x).$$
Any $f$ satisfying
$$
[L_{\la}f](x) = r\beta_1 f(0)
$$
is a solution of the equation
$$
[L_{\la}f](x) = F,
$$
where $F$ is a constant, and thus it is of the form
$$
f(x) = C_1f_1(x)+C_2f_2(x) - \frac{F}{a}
$$
where $f_1$ and $f_2$ are arbitrary solutions to $L_{\la}[f]=0$. We see that we can take $f_1(x) =(\la+ax)$ and, as in \eqref{secsol}, here we obtain
$$
f_2(x) = (\la+ax)\int_0^x \frac{e^{\frac{1}{2ar}(\la + a\xi)^2d\xi}}{(\la+a\xi)^2}d\xi.
$$
We see that $f_2$ grows faster than any polynomial as $x\to \infty$ so that, in $X_m^*$,
$$
f(x) = C_1(\la+ax) - \frac{F}{a}
$$
where $F$ and $C_1$  are related by
$$
[L_\la f](x) = F = r\beta_1 f(0) = r\beta_1 \left(C_1 \la -\frac{F}{a}\right).
$$
Hence
$$
C_1 = \frac{F\alpha_1}{a\la r\beta_1}
$$
and
$$
f(x) = \frac{F}{a\la r\beta_1} \left(\alpha_1(\la +a x)-\la r\beta_1\right) =  \frac{F}{\la r\beta_1}(\la +\alpha_1 x)
$$
Thus, the only (up to a multiplicative constant) eigenfunction of \eqref{ep1} in $X_m^*$ is given by $w(x) = \la+ \alpha_1 x$. Returning with such a $w$ to
\eqref{eigprob1pa}, we find that $w$ solves it if and only if $\lambda$ satisfies \eqref{chareq}.
Since we are looking
for positive solutions, it follows that $s_0=\lambda_+$ and
\begin{equation}\label{eigpair2}
(s_0, w_0(x)) :=
\left(\lambda_+, \sigma (\alpha_1 x+ \lambda_+) \right),\quad
\sigma = \frac{1}{\lambda_+ - \lambda_-},
\end{equation}
is the positive left eigenpair of $(K_{\beta,m}, D(Z_{\beta,m}))$, normalized so as  $\langle w_0, v_0\rangle = 1$. Hence, $\mathcal{P} u = \langle w_0, u\rangle v_0$ and, on the account
of Corollary~\ref{asyexpgr}, we have
\begin{equation}\label{asyexp1}
e^{-\lambda_+ t} u(x,t)  = \langle w_0, u_0\rangle v_0(x) + o(1)
= \left[\frac{\lambda_+}{\lambda_+-\lambda_-}M_0(0)
+\frac{\lambda_+\lambda_-}{r(\lambda_--\lambda_+)}M_1(0)\right]v_0(x) + o(1),
\end{equation}
in $X_m$, $m>1$, uniformly for large values of $t>0$.

\begin{figure}
\begin{center}
\includegraphics[width=2in]{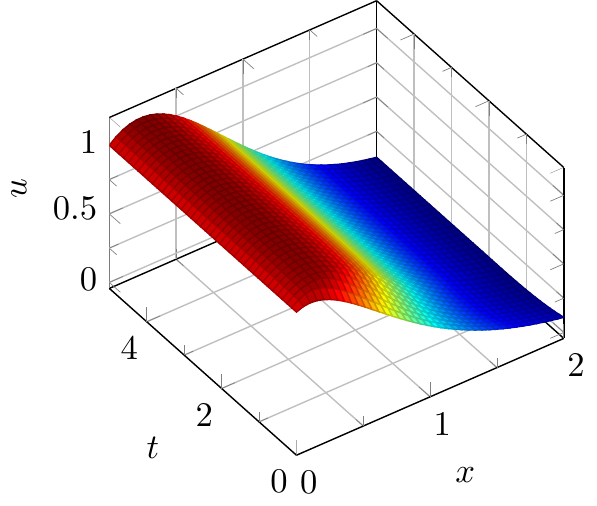} \includegraphics[width=2in]{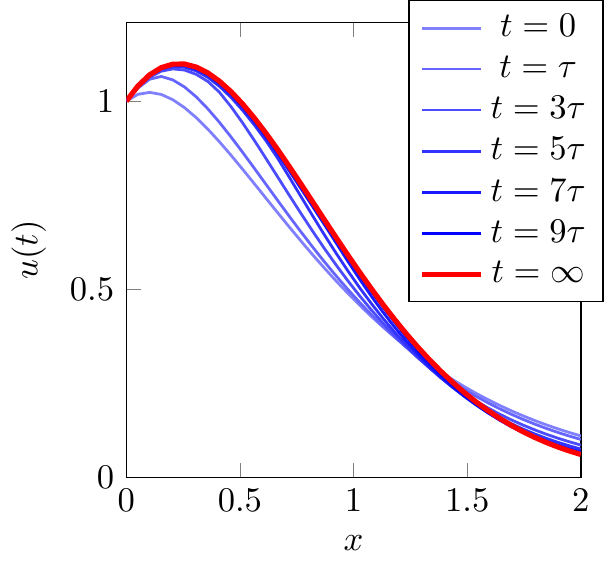} \\
\includegraphics[width=2in]{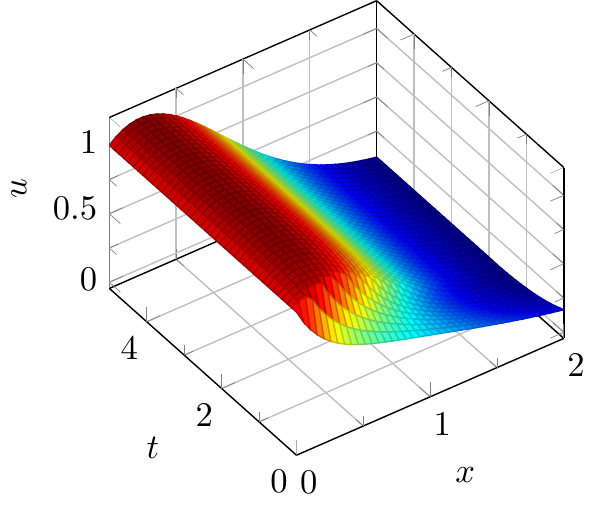} \includegraphics[width=2in]{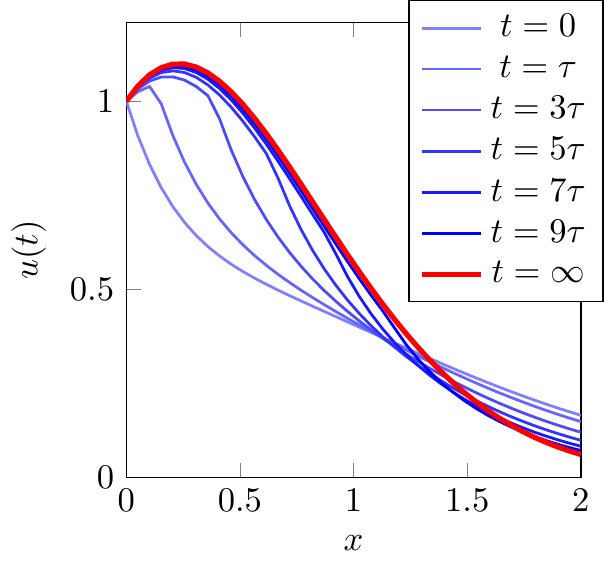} \\
 \caption{Large time asymptotics of \eqref{Main1}: normalized semigroup solutions governed
by \eqref{exsol} (left column); snapshots of $e^{-\lambda_+t}u(t)$, with $\tau = \frac{1}{8}$
(right column).
}\label{fig2}
\end{center}
\end{figure}

To illustrate the asymptotic formula, we compare dynamics governed by \eqref{exsol} and by \eqref{asyexp1}
for large values of $t>0$. For the sake of simplicity, we take
\[
a = r  =1,\quad \beta_0 = \beta_1 = \frac{1}{2},\quad u_0(x) = \Bigl(\frac{5}{2}x+1\Bigr)e^{-2x}.
\]
In these settings,
\[
\lambda_+ = \frac{3}{2},\quad \lambda_ - = -1,\quad u_0(0) = M_0(0) + M_1(0) = 1,
\]
so that $u_0 \in D(K_{\beta,m})$, $m>1$, and the semigroup solution \eqref{exsol} is regular in the spatial domain
$\mathbb{R}_+$.
The normalized exact solution $e^{-\lambda_+t}u(x,t)$ of \eqref{exsol} and its snapshots
at $t=0$ and at $t=(2i+1)\tau$, $0\le i\le 4$, with $\tau = \frac{1}{8}$, are plotted in the top-left and
the top-right diagrams of Fig.~\ref{fig2}, respectively. In agreement with \eqref{asyexp1}, after a short transition
stage the solution settles to it asymptotic profile $\langle w_0, u_0\rangle v_0(x)$, which is shown in red
in the top-right diagram of Fig.~\ref{fig2}. The asymptotic profile does not depend on the particular shape
of the initial data. To illustrate this point, we repeat our calculations, but this time for
\[
u_0(x) = \Bigl(2x^2+1\Bigr)e^{-2x},
\]
see the two bottom diagrams in Fig.~\ref{fig2}.  As postulated by the asynchronous exponential growth property, there are no qualitative changes in the large time
behavior of $e^{-\lambda_+t}u(x,t)$. As $t$ increases, the normalized solution approaches the same
limit $\langle w_0, u_0\rangle v_0(x)$ (shown in red) as in the previous example.

\section{Conclusion}\label{secconc}
In the paper, we discussed global well-posedness and large-time asymptotics of the growth-fragmentation equation
equipped with unbounded transport and fragmentation rates and with non-local McKendrick--von Foerster
boundary condition. The results of this paper build upon the well-posedness theory in $X_m, m>1,$ developed in \cite{Lamb2020} for \eqref{Tsapele} with homogeneous Dirichlet boundary conditions ($\beta =0$). The novelty of this paper lies in extending the well-posedness theory, developed in \cite{BL09} only for $m=1$, to $X_m$ with $m>1,$ and also in generalizing the long-time asymptotics results, obtained in \cite{MKJB} for $\beta =0$, to the general McKendrick--von Foerster setting. Our approach is based on the observation that in an appropriate functional setting,
the semigroup, governing solution to the non-local model, can be realized as a compact perturbation of the
semigroup associated with the same model but equipped with homogeneous Dirichlet boundary data.
This significantly simplified formal analysis and allowed us to transfer results available for the former
problem directly to our non-local settings. In particular, we demonstrated that the very recent spectral gap theory of
\cite{MKJB}, coupled with new perturbation results and complete characterization of the irreducibility of involved semigroups in Sections~\ref{secGT} and
\ref{secsg}, respectively, yields a complete description of the long time dynamics of the full model
under mild restrictions on its coefficients.

 In the last section, we found explicit solutions and the Perron eigenpair of a toy growth--fragmentation models \eqref{Main1} that, nevertheless, is considered in applications. We could see that finding large-time behaviour of solutions by direct methods is an extremely tedious exercise. In this sense, the irreducibility and the spectral gap theories
of Sections~\ref{secsg}--\ref{subsecsg} provide
a powerful tool that, in a good number of practical cases, yields a straightforward description
of the growth--fragmentation dynamics for large values of $t>0$.


\begin{thebibliography}{10}

\bibitem{abr}
M.~Abramowitz and I.~A. Stegun.
\newblock {\em Handbook of mathematical functions: with formulas, graphs, and
  mathematical tables}, volume~55.
\newblock Courier Corporation, 1964.

\bibitem{Apo}
T.~M. Apostol.
\newblock {\em Mathematical analysis: a modern approach to advanced calculus}.
\newblock Addison-Wesley Publishing Co., Inc., Reading, Mass., 1957.

\bibitem{Nag}
W.~Arendt, A.~Grabosch, G.~Greiner, U.~Groh, H.~P. Lotz, U.~Moustakas,
  R.~Nagel, F.~Neubrander, and U.~Schlotterbeck.
\newblock {\em One-parameter semigroups of positive operators}, volume 1184 of
  {\em Lecture Notes in Mathematics}.
\newblock Springer-Verlag, Berlin, 1986.

\bibitem{JA}
J.~Banasiak and L.~Arlotti.
\newblock {\em Perturbations of positive semigroups with applications}.
\newblock Springer Monographs in Mathematics. Springer-Verlag London, Ltd.,
  London, 2006.

\bibitem{BJSJEE}
J.~Banasiak, L.~O. Joel, and S.~Shindin.
\newblock Discrete growth-decay-fragmentation equation: well-posedness and
  long-term dynamics.
\newblock {\em J. Evol. Equ.}, 19(3):771--802, 2019.

\bibitem{BL09}
J.~Banasiak and W.~Lamb.
\newblock Coagulation, fragmentation and growth processes in a size structured
  population.
\newblock {\em Discrete Contin. Dyn. Syst. Ser. B}, 11(3):563--585, 2009.

\bibitem{Lamb2020}
J.~Banasiak and W.~Lamb.
\newblock Growth-fragmentation-coagulation equations with unbounded coagulation
  kernels.
\newblock {\em Philos. Trans. Roy. Soc. A}, 378(2185):20190612, 22, 2020.

\bibitem{JWL}
J.~Banasiak, W.~Lamb, and P.~Lauren\c{c}ot.
\newblock {\em Analytic methods for coagulation-fragmentation models. {V}ol.
  {I}}.
\newblock Monographs and Research Notes in Mathematics. CRC Press, Boca Raton,
  FL, [2020] \copyright 2020.

\bibitem{BPR}
J.~Banasiak, K.~Pich\'{o}r, and R.~Rudnicki.
\newblock Asynchronous exponential growth of a general structured population
  model.
\newblock {\em Acta Appl. Math.}, 119:149--166, 2012.

\bibitem{Banasiak2022}
J.~Banasiak, D.~W. Poka, and S.~Shindin.
\newblock Explicit solutions to some fragmentation equations with growth or
  decay.
\newblock {\em Journal of Physics A: Mathematical and Theoretical},
  55(19):194001, 2022.

\bibitem{BG1}
E.~Bernard and P.~Gabriel.
\newblock Asymptotic behavior of the growth-fragmentation equation with bounded
  fragmentation rate.
\newblock {\em J. Funct. Anal.}, 272(8):3455--3485, 2017.

\bibitem{BG2}
E.~Bernard and P.~Gabriel.
\newblock Asynchronous exponential growth of the growth-fragmentation equation
  with unbounded fragmentation rate.
\newblock {\em J. Evol. Equ.}, 20(2):375--401, 2020.

\bibitem{BertW}
J.~Bertoin and A.~R. Watson.
\newblock A probabilistic approach to spectral analysis of growth-fragmentation
  equations.
\newblock {\em J. Funct. Anal.}, 274(8):2163--2204, 2018.

\bibitem{BTK}
W.~Biedrzycka and M.~Tyran-Kami\'{n}ska.
\newblock Self-similar solutions of fragmentation equations revisited.
\newblock {\em Discrete Contin. Dyn. Syst. Ser. B}, 23(1):13--27, 2018.

\bibitem{CGY}
J.~A. Ca\~{n}izo, P.~Gabriel, and H.~Yolda\c{s}.
\newblock Spectral gap for the growth-fragmentation equation via {H}arris's
  theorem.
\newblock {\em SIAM J. Math. Anal.}, 53(5):5185--5214, 2021.

\bibitem{Diek}
O.~Diekmann, H.~J. A.~M. Heijmans, and H.~R. Thieme.
\newblock On the stability of the cell size distribution.
\newblock {\em J. Math. Biol.}, 19(2):227--248, 1984.

\bibitem{DJG}
M.~Doumic~Jauffret and P.~Gabriel.
\newblock Eigenelements of a general aggregation-fragmentation model.
\newblock {\em Math. Models Methods Appl. Sci.}, 20(5):757--783, 2010.

\bibitem{Klaus}
K.~Engel and R.~Nagel.
\newblock {\em One-parameter Semigroups for Linear Evolution Equations}.
\newblock Springer, Berlin, 2000.

\bibitem{Klaus1}
K.~Engel and R.~Nagel.
\newblock {\em Short course on operator semigroups}.
\newblock Springer Science + Business Media, LCC, New York, 2006.

\bibitem{ian}
M.~Iannelli.
\newblock Mathematical theory of age-structured population dynamics.
\newblock {\em Giardini editori e stampatori in Pisa}, 1995.

\bibitem{Ka}
T.~Kato.
\newblock Perturbation theory for nullity, deficiency and other quantities of
  linear operators.
\newblock {\em J. Analyse Math.}, 6:261--322, 1958.

\bibitem{IVo}
I.~Marek.
\newblock Frobenius theory of positive operators: Comparison theorems and
  applications.
\newblock {\em SIAM Journal on Applied Mathematics}, 19(3):607--628, 1970.

\bibitem{MMP}
P.~Michel, S.~Mischler, and B.~Perthame.
\newblock General relative entropy inequality: an illustration on growth
  models.
\newblock {\em J. Math. Pures Appl. (9)}, 84(9):1235--1260, 2005.

\bibitem{Sch}
S.~Mischler and J.~Scher.
\newblock Spectral analysis of semigroups and growth-fragmentation equations.
\newblock {\em Ann. Inst. H. Poincar\'{e} C Anal. Non Lin\'{e}aire},
  33(3):849--898, 2016.

\bibitem{MKM2}
M.~Mokhtar-Kharroubi.
\newblock On the convex compactness property for the strong operator topology
  and related topics.
\newblock {\em Math. Methods Appl. Sci.}, 27(6):687--701, 2004.

\bibitem{MKM1}
M.~Mokhtar-Kharroubi.
\newblock Compactness properties of perturbed sub-stochastic {$C_0$}-semigroups
  on {$L^1(\mu)$} with applications to discreteness and spectral gaps.
\newblock {\em M\'{e}m. Soc. Math. Fr. (N.S.)}, (148):iv+87, 2016.

\bibitem{MKJB}
M.~Mokhtar-Kharroubi and J.~Banasiak.
\newblock On spectral gaps of growth-fragmentation semigroups in higher moment
  spaces.
\newblock {\em Kinet. Relat. Models}, 15(2):147--185, 2022.

\bibitem{PR}
B.~Perthame and L.~Ryzhik.
\newblock Exponential decay for the fragmentation or cell-division equation.
\newblock {\em J. Differential Equations}, 210(1):155--177, 2005.

\bibitem{PZ2003}
A.~D. Polyanin and V.~F. Zaitsev.
\newblock {\em Handbook of exact solutions for ordinary differential
  equations}.
\newblock Chapman \& Hall/CRC, Boca Raton, FL, second edition, 2003.

\bibitem{Schl}
G.~Schl\"{u}chtermann.
\newblock On weakly compact operators.
\newblock {\em Math. Ann.}, 292(2):263--266, 1992.

\bibitem{Webb}
G.~F. Webb.
\newblock {\em Theory of nonlinear age-dependent population dynamics},
  volume~89 of {\em Monographs and Textbooks in Pure and Applied Mathematics}.
\newblock Marcel Dekker, Inc., New York, 1985.

\end{thebibliography}

\end{document}